\documentclass[11pt]{article}
\usepackage{amssymb}
\usepackage{amsmath}
\def\.{\partial_t}
\def\:{\partial_{tt}}

\def\al{\alpha}
\def\e{\varepsilon}
\def\eps{\varepsilon}
\def\f{\varphi}
\def\l{\lambda}
\def\la{\lambda}

\def\s{\sigma}
\def\t{\tau}

\def\La{\Lambda}

\def\G{\Gamma}
\def\ga{\gamma}
\def\vka{\varkappa}
\def\vr{\varrho}

\def\fA{\mathfrak{A}}
\newcommand{\fB}{{\mathfrak{B}}}

\def\cA{\mathcal{A}}
\def\cB{\mathcal{B}}
\def\cD{\mathcal{D}}
\def\sD{\mathcal{D}}
\def\cK{\mathcal{K}}
\def\cL{\mathcal{L}}
\def\cM{\mathcal{M}}
\def\cE{\mathcal{E}}

\def\cR{\mathcal{R}}
\def\H20{\mathcal{H}_{2,0}}
\def\R{\mathbb{R}}

\def\cN{\mathcal{N}}
\def\cH{\mathcal{H}}

\def\<{\langle}
\def\>{\rangle}
\def\8{\infty}

\def\hf{\frac{1}{2}}

\def\wrt{with respect to }

\oddsidemargin
0.1cm \textwidth 16.2cm

\newtheorem{lemma}{Lemma}[section]
\newtheorem{theorem}[lemma]{Theorem}
\newtheorem{remark}[lemma]{Remark}
\newtheorem{proposition}[lemma]{Proposition}

\newtheorem{definition}[lemma]{Definition}
\newtheorem{assumption}[lemma]{Assumption}

\newcommand{\dist}{{\operatorname{dist}}}

\newenvironment{declaration}[1]{\trivlist
\item[\hskip \labelsep{\bf #1 }]\ignorespaces}{\endtrivlist}
\newenvironment{proofof}[1]{\begin{declaration}{#1}}{\hfill
$\square$\end{declaration}}
\newenvironment{proof}{\begin{proofof}{Proof.}}{\end{proofof}}

\begin{document}

\title{Synchronization in coupled second order in time infinite-dimensional models }
\author{Igor
Chueshov\thanks{e-mail:
chueshov@karazin.ua}  
\\ \\
Department of Mechanics and Mathematics, \\ Karazin Kharkov
National University, \\ Kharkov, 61022,  Ukraine
 }
 \maketitle

\begin{abstract}
We study  asymptotic synchronization  at the level of global attractors in a class of coupled  second order in time models which arises in dissipative wave and elastic structure dynamics.
Under some conditions we prove that this synchronization arises in the infinite coupling intensity limit and show that for identical subsystems this phenomenon appears for finite intensities. 
 Our argument involves a 
   method based on ``compensated'' compactness and
quasi-stability estimates. 
As an application we consider the
nonlinear Kirchhoff, Karman and Berger plate models with  different
types of boundary conditions. Our results
can be also applied to the nonlinear wave equations in an arbitrary
dimension.
We  consider  synchronization in  sine-Gordon  type models
 which describes distributed Josephson junctions. 

\end{abstract}

\section{Introduction}

Our goal in this paper is  to study long-time dynamics of a class of coupled systems consisting of two second order in time evolution equations. These systems are abstract models for studies of elastic and wave dissipative dynamics  in different situations. Under some set of hypotheses concerning the model we
 first prove the existence of finite dimensional global attractors
and study their dependence on interaction operators.
 Then we apply  these results to analyze synchronization phenomena.
 In this paper we understand these phenomena at the level of global attractors. This means that in  synchronized regime the attractor of coupled system  
 becomes ``diagonal'' in some sense.
 \par 
 Our main results are presented in Theorems \ref{th:att-sync} and \ref{th:sync-ident}. In particular, Theorem \ref{th:att-sync} proves asymptotic synchronization in the limit of large coupling and Theorem \ref{th:sync-ident}  dealing with interaction of identical 
 systems shows that synchronization is possible for finite values of coupling intensity parameter. 
 As a preliminary step we obtain a result on uniform dissipativity \wrt coupling intensity parameters (see Theorem \ref{th:-u-dis}). 
 We also discuss possibility of synchronization in infinite-dimensional systems by means of finite-dimensional interaction operators.
 As application of these results we consider  a range of nonlinear elastic plate models and also wave dynamics of different types.

 We note that
recently the subject of synchronization of coupled (identical
or not) systems has received considerable attention. There are now
quite a few monographs \cite{balanov,leonov2,majstrenko,osipov,Strog,wu} in this field,
which contain extensive lists of references. In the case of infinite
dimensional systems the synchronization problem has been studied in
\cite{CCK07,Car-Prim-sync2000,CRD,Hale-jdde97,H}  for  coupled  parabolic
systems. Synchronization in Berger plates
(they are a particular case of our abstract models)  was considered in \cite{naboka07,naboka08,naboka09}. 
Master-slave synchronization of coupled parabolic-hyperbolic PDE systems was considered in \cite{Chu04,Chu07}.
 The methods involved in these publications 
 relies either on the parabolic regularization (see \cite{CCK07,CRD,Hale-jdde97}) or on a special structure of the model (see \cite{naboka07,naboka08,naboka09}).
 Our approach is different.
 \\par 
 As an important technical tool 
we involve  the   method developed in \cite{cl-mem} (see also
also \cite[Chapter 8]{cl-book} and \cite{Chu-dqsds15}) based on an observability-type
estimate  which allows us  to establish uniform quasi-stability estimates 
in the case of critical nonlinearities (such as in the von Karman
and Berger models).
In the standard way (see, e.g., \cite{Chu-dqsds15}, \cite{cl-mem} or
\cite[Chapter 8]{cl-book}) these quasi-stability estimates lead
 to appropriate uniform bounds for  attractors which are important for asymptotic synchronization.
\par
The paper is organized as follows.
The next Section~\ref{sect2} is devoted to
 preliminary  considerations.  
 We describe here our abstract model, formulate main hypotheses and provide
  well-posedness result for rather general 
  situation.
Section~\ref{sect4} contains our main results on attractors and synchronization.
In  Section~\ref{Sect-appl} we discuss possible applications.
The Appendix  contains 
 some general facts
from the dissipative dynamical systems theory.

\section{Preliminaries}
\label{sect2}
In this section we describe the problem and state our basic notations and hypotheses.
Then we provide a well-posedness theorem which is easily derived from the known results.

\subsection{Abstract model and main hypotheses} In  a  Hilbert space
$H$  we deal with following coupled equations
\begin{subequations}\label{wave-eq1}
 \begin{equation}\label{wave-eq1a}
    u_{tt} +\nu_1 A u+ D_{11}u_t +D_{12}v_t
    + K_{11}u+ K_{12}v +B_1(u)=0,
 \end{equation}
 \begin{equation}\label{wave-eq1b}
    v_{tt} +\nu_2 A v+ D_{21}u_t +D_{22}v_t
    + K_{21}u+ K_{22}v +B_2(v)=0,
 \end{equation}
 equipped with initial data
 \begin{equation}
 u(0)=u_0,~~u_t(0)=u_1,~~  v(0)=v_0,~~v_t(0)=v_1,
\end{equation}
 \end{subequations}
 under the following set of hypotheses.
\begin{assumption}\label{A1-sync2}
{\rm
\begin{enumerate}
  \item[(i)] $A$ is a self-adjoint positive operator densely
defined on a domain $\cD(A)$ in a separable Hil\-bert space $H$, $\nu_1,\nu_2>0$ are parameters. We  assume that
the resolvent of $A$ is compact in $H$. This implies that there is  orthonormal basis $\{e_k\}$ in $H$ consisting of the eigenvectors
of the operator $A$ :
\[
A e_k=\lambda_k e_k,
\quad 0<\lambda_1\le \lambda_2\le\cdots, \quad \lim_{k\to\infty}\lambda_k =\infty.
\]
We denote by $\|\cdot \|$
and $( \cdot, \cdot )$ the norm and
the scalar product in $H$.  We  also denote
by $H^s$ (with $s>0$) the domain $\sD(A^{s})$ equipped with the graph norm
$\|\cdot \|_s=\|A^{s}\cdot\|$.  $H^{-s}$ denotes the completion
of $H$ with respect to the norm $\|\cdot \|_{-s}=\|A^{-s}\cdot\|$. The symbol $(\cdot,\cdot)$  denotes not only
the scalar product but also  the duality between $H^{s}$ and $H^{-s}$.  
Below we also use the notation $\bar{H}^s=H^s\times H^s$.
  \item[(ii)] 
  The damping operators $D_{ij}\, :\, H^{1/2}\mapsto H^{-1/2}$ are linear mappings such that the matrix operator
  \[
  \sD =\left(\begin{matrix}
  D_{11}& D_{12} \\
 D_{21}  & D_{21}
\end{matrix}\right)\; : ~ H^{1/2}\times H^{1/2}\mapsto H^{-1/2}\times H^{-1/2}
\]
generates a symmetric  
nonnegative bilinear form on 
$\bar{H}^{1/2}\equiv H^{1/2}\times H^{1/2}$.  
  \item[(iii)] The interaction operators 
   $K_{ij}\,:\, H^{1/2}\mapsto  H$ 
 are  linear and  
  \[
  \cK =\left(\begin{matrix}
  K_{11}& K_{12} \\
 K_{21}  & K_{21}
\end{matrix}\right)\; : ~ H^{1/2}\times H^{1/2}\mapsto H\times H
\]
generates a symmetric 
nonnegative
  bilinear form on 
$\bar{H}^{1/2}$.  
  \item[(iv)]
The nonlinear operators $B_i: H^{1/2} \to H $ are
locally Lipschitz, i.e., 
for every $\varrho>0$
there exists a constant $L(\vr)$ 
such that
\begin{equation*}
\| B_i(u)-B_i(v)\|\le L(\vr) \|u-v\|_{1/2},~~i=1,2, 
\end{equation*}
for all $u,v\in H^{1/2}$ such that $  \Vert u\Vert_{1/2}, \Vert v\Vert_{1/2} \le\vr$.
In addition we assume that $ B_i(u)$ are potential operators.
This means that 
$ B_i(u)=\Pi_i'(u)$,
 where  $\Pi_i\,:\, H^{1/2}\mapsto  \R $
is a Frech\'{e}t differentiable  functional on $H^{1/2}$, i.e., 
 \begin{equation*}
 \lim_{\|v\|_{1/2}\to0}\frac1{\|v\|_{1/2}}\Big[\Pi_i(u+v)-\Pi_i(u)-(\Pi_i'(u),v)\Big]=0.
 \end{equation*}
We also assume that 
 $\Pi_i(u)=\Pi_{0i}(u)+\Pi_{1i}(u)$,
 where $\Pi_{0i}(u)$
 is a nonnegative   locally bounded functional on $H^{1/2}$ and 
\begin{equation*}
\forall\, \eta>0\; \exists\, C_\eta :~~
|\Pi_{1i}(u)|\le \eta \big[\|A^{1/2}u\|^2 +\Pi_{0i}(u)\big]+C_\eta,~~ u\in H^{1/2}.
\end{equation*}
\end{enumerate}
}
\end{assumption}
\par
The problem in \eqref{wave-eq1} can be written as
\begin{equation}\label{ver-problem}
   U_{tt} +\cA U+{\cal D} U_t+ \cK U +{\cal B}(U)=0,~~ U(0)=U_0,~~U_t(0)=U_1,
\end{equation}
where the operators $\cD$ and $\cK$ are defined above and
\[
U=\left(\begin{matrix}
  u \\
  v
\end{matrix}\right),~~~
\cA = \left(\begin{matrix}
  \nu_1 & 0 \\
  0 &\nu_2
\end{matrix}\right) A,~~~
\cB(U)=\left(\begin{matrix}
  B_1(u) \\
  B_2(v)
\end{matrix}\right).
\]
As it was already mentioned in Introduction our main motivation for \eqref{wave-eq1} (and \eqref{ver-problem}) 
and also for the hypotheses in Assumption \ref{A1-sync2} is
 related with applications to coupled plate and wave systems (see Section \ref{Sect-appl}).

\subsection{Well-posedness}\label{sect3}

Now we consider the existence and uniqueness
theorem for \eqref{ver-problem}. We start with adaptation of the standard  definition (see, e.g, \cite{cl-mem} and the references therein) to our model.

\begin{definition}\label{str-sol-2ord}
{\rm 
A function $U(t)\in C([0,T]; \bar{H}^{1/2})\cap C^1([0,T]; \bar{H})$
possessing the properties $U(0)=U_0$ and $U_t(0)=U_1$
is said to be
\begin{enumerate}
  \item[{\bf (S)}] {\em strong solution} to
problem (\ref{ver-problem})  on the interval $[0,T]$, iff
\begin{itemize}
  \item $u\in W_1^1(a,b; \bar{H}^{1/2})$ and
$u_t\in W_1^1(a,b; \bar{H})$    
for any $0<a<b<T$,
where
\begin{equation*}
W_1^1(a,b;H)=\left\{ f\in C(a,b;H)\; :\; f^\prime\in
L_1(a,b;H)\right\},
\end{equation*}
\item
$ \cA U(t) +   D U_t(t) \in  \bar{H}$ for almost all $t\in [0,T]$;
  \item equation in (\ref{ver-problem}) is satisfied in $\bar H$
  for almost all $t\in [0,T]$;
\end{itemize}
  \item[{\bf (G)}]
{\em generalized solution} to problem~(\ref{ver-problem})
on the interval $[0,T]$, iff
 there exists a sequence
  of strong solutions $\{ U_n(t)\}$ 
 with initial data
$(U_{0n}, U_{1n})$ instead of $(U_{0}, U_{1})$
such that
\begin{equation*}
\lim_{n\to\infty}\max_{t\in[0,T]}\left\{ \|\partial_t U(t)-\partial_tU_n(t)\|+
 \|\cA^{1/2}( U(t)-U_n(t))\|\right\}= 0.
\end{equation*} 
\end{enumerate}
}
\end{definition}
Application of Theorem~1.5 from \cite{cl-mem}
gives the following   well-posedness
result.
\begin{theorem}\label{t2.wp-smp}
Let $T > 0 $ be arbitrary.
Under Assumption~\ref{A1-sync2} the following statements hold.
\begin{itemize}
\item{\bf Strong solutions:}
For every
$(U_0; U_1) \in  \bar{H}^{1/2}\times\bar{H}^{1/2}$,
such that $\cA U_0+  \cD U_1\in \bar H$
there exists   unique strong solution to problem (\ref{ver-problem})
 on the interval
$[0,T]$ such that
\begin{equation*}\label{8.1.3a}
\begin{array}{c}
(U_t; U_{tt}) \in L_\infty( 0, T; \bar{H}^{1/2} \times \bar H),
\quad U_{t} \in C_r([0,T);   \bar{H}^{1/2}),
\\ \\
U_{tt} \in C_r([0,T); \bar H )\quad\mbox{and}\quad
\cA U(t)+ \cD U_t(t)\in C_r([0, T); \bar{H}),
\end{array}
\end{equation*}
where we denote by $C_r$  the space of right continuous
functions. This solution satisfies the energy relation
\begin{equation}\label{8.1.4}
\cE(U(t), U_t(t))+\int_0^t (\cD U_t(\tau), U_t(\tau)) d\tau
= 
\cE(U_0, U_1),
\end{equation}
where the energy $\cE$ is defined by the relation
\[
\cE (U_0,U_1)=\cE_1 (u_0,u_1)+\cE_2 (v_0,v_1)
+\cE_{int}(u_0,v_0),
\]
with $U_0=(u_0;v_0)$, $U_1=(u_1;v_1)$,
\begin{equation*}
\cE_i(u_0, u_1)=E_i(u_0, u_1)+\Pi_i(u_0)
\equiv\frac12\left( \|u_1\|^2 + \nu_i\|A^{1/2} u_0\|^2\right)+\Pi_i(u_0).
\end{equation*}
and 
\begin{equation*}
\cE_{int}(u_0, v_0)=\hf (\cK U_0, U_0).
\end{equation*}
\item{\bf Generalized solutions:}
For every  $(U_0;U_1) \in \bar{H}^{1/2}\times \bar H$
there exists unique generalized solution.
This solution possesses the property 
$\cD^{1/2} u_t\in L_2(0,T; \bar{H})$ and
satisfies  the energy  inequality 
\begin{equation}\label{8.1.4new}
\cE(U(t), U_t(t))+\int_0^t (\cD U_t(\tau), U_t(\tau)) d\tau
\le  
\cE(U_0, U_1).
\end{equation}
\end{itemize}
If $U^1$ and $U^2$ are 
 generalized solutions 
 with different initial data and $Z=U^1-U^2$,
 then
\begin{equation*}
\|Z_t(t)\|^2+\|\cA^{1/2} Z(t)\|^2
+\|\cK^{1/2} Z(t)\|^2
\le\left(\|Z_t(0)\|^2+\|\cA^{1/2} Z(0)\|^2 +\|\cK^{1/2} Z(t)\|^2\right) e^{a_R t}
\end{equation*}
 provided 
 $
\|U^i_t(0)\|^2+\|\cA^{1/2} U^i(0)\|^2
+\|\cK^{1/2} U^i(0)\|^2\le R^2$.
\end{theorem}

By Theorem \ref{t2.wp-smp} problem \eqref{ver-problem} generates a dynamical system $(\cH, S_t)$ in the space
\[
\cH=\bar{H}^{1/2}\times \bar{H}\equiv  H^{1/2}\times H^{1/2}\times H\times H
\]
with  evolution operator defined by the relation
\[
S_t(U_0;U_1)=(U(t);U_t(t)),
\]
where $U(t)$ is a generalized
solution to problem \eqref{ver-problem}.

\section{Global Attractors}\label{sect4}
In this section we prove the existence of a global attractor
for the dynamical system $(\cH, S_t)$
and study its properties.
Keeping in  mind further application we 
impose additional hypotheses concerning
the damping operator $\cD$ and the source term 
$\cB$.
\begin{assumption}\label{as:a-comp}
{\rm 
Let Assumption~\ref{A1-sync2} be in force and 
 \begin{enumerate}
  \item[(i)] $\cD$ is strictly positive, i.e. exists $c_0>0$ such that
 \begin{equation*}\label{d-positive}
  (\cD W,W)\ge c_0\|W\|^2,~~~W\in \bar H^{1/2};
\end{equation*}
 \item[(ii)] \textbf{ either} $B_i$ are subcritical,
  i.e.,
for every $\varrho>0$
there exists a constant $L(\vr)$ 
such that
\begin{equation*}
\| B_i(u)-B_i(v)\|\le L(\vr) \|u-v\|_{1/2-\delta},~~i=1,2, ~~\delta>0,
\end{equation*}
for all $u,v\in H^{1/2}$ such that $  \Vert u\Vert_{1/2}, \Vert v\Vert_{1/2} \le\vr$,
or \textbf{else}  the potential energies $\Pi_i$ are continuous 
  on $H^{1/2-\delta}$ for some $\delta>0$ and
 the mapping $u\mapsto A^{-l}B_i(u)$ is continuous from $H^{1/2-\delta}$ 
 into $H^{-l}$ for some $\delta,l>0$, $i=1,2$.
  \end{enumerate}     	
}
\end{assumption}
\begin{proposition}\label{th:-a-comp}
Let Assumption \ref{as:a-comp} be in force. Then the system  
$(\cH,S_t)$ generated by problem \eqref{ver-problem} is asymptotically smooth
(see the definition in the Appendix). 	
\end{proposition}
\begin{proof}
	We apply Theorem 3.26 and Proposition 3.36
	from \cite{cl-mem}. 
	Since
\begin{align*}
	|( \cD V,W)|\le & [(\cD V,V)]^{1/2}[(\cD W,W)]^{1/2} \le C [(\cD V,V)]^{1/2}\|\cA ^{1/2}W\| \\ \le &
	 C_\eps (\cD V,V)+ \eps \|\cA^{1/2} W\|^2,
\end{align*}
relation (3.60) in \cite[p.54]{cl-mem} obviously holds 
	in a simplified  form. 	
\end{proof}

One can see from the energy inequality in \eqref{8.1.4new}
that the system $(\cH,S_t)$ is gradient
with the full energy $\cE(U_0;U_1)$ as a strict Lyapunov function (see the Appendix for the corresponding definitions).
Therefore
by Corollary~2.29 \cite{cl-mem} 
(see Theorem \ref{Theorem 2.2.} and Remark \ref{re:th2.2} in the Appendix) 
 to guarantee the existence of a global attractor we need 
boundedness of equilibria.
This leads to the following assertion.

\begin{theorem}\label{th:atr-setN}
Let Assumptions \ref{as:a-comp} be in force. Assume in addition
that there exist  $\nu <\nu_i$
 and $C\ge 0$ such that
\begin{equation}\label{8.1.1c1a}
 \nu \| A^{1/2}u\|^2 +(B_i(u),u)+C \ge 0\;,\quad u\in H^{1/2},~~i=1,2.
\end{equation}
Then the system 
$(\cH,S_t)$ generated by problem \eqref{ver-problem} possesses a compact global attractor.	
\end{theorem}
\begin{proof}
Stationary solutions $U=(u;v)\in \bar{H}^{1/2}$ solve the problem
 \begin{equation*}
    \nu_1 A u 
    + K_{12}u+ K_{12}v +B_1(u)=0,
 \end{equation*}
 \begin{equation*}
    \nu_2 A v + K_{21}u+ K_{22}v +B_2(v)=0.
 \end{equation*}
Using the multipliers $u$ for the first equation and $v$ for the second and also positivity of the operator $\cK$, we obtain that
\[
 \nu_1 \|\cA^{1/2}u\|^2  + \nu_2 \|\cA^{1/2}v\|^2 +(B_i(u),u)+(B_i(v),v) \le 0
 \]
By \eqref{8.1.1c1a} this yields 
$  \|\cA^{1/2}u\|^2  +  \|\cA^{1/2}v\|^2\le C$
(with $C$ independent of $\cD$ and $\cK$). Thus stationary solutions are bounded. 
\end{proof}

\subsection{Uniform dissipativity}

For synchronization phenomena it is important 
to have bounds for the attractor independent of interaction operators $\cD$ and $\cK$.
In spite of the set of stationary 
solutions is uniformly bounded \wrt $\cD$ and $\cK$ Theorem~\ref{th:atr-setN} does not provide appropriate bounds for the attractor. Below we use an approach based  Lyapunov type functions 
which allows us to prove uniform dissipativity  
 of the system $(\cH,S_t)$.
To simplify argument it is convenient to introduce intensity factors $\alpha$ and $\vka$ for interactions in velocities and displacements. 
Moreover, we assume a particular structure of $\cD$ related with interaction operator $\cK$.
So instead of \eqref{ver-problem}
we consider
\begin{equation}\label{ver-problem-kappa}
   U_{tt} +\cA U+(\cD_0+ \alpha {\cal K}) U_t+ \vka \cK U +{\cal B}(U)=0,~~ U(0)=U_0,~~U_t(0)=U_1,
\end{equation}
where $\alpha$ and $\vka$ are nonnegative parameters.

In addition to Assumption~\ref{A1-sync2}
we impose the following hypotheses. 
\begin{assumption}\label{as:u-diss}
\begin{enumerate}
  \item[(i)] The operator $\cD_0$ is bounded from $\bar H^{1/2}$ into $\bar H$ and  there exist  $c_0>0$ and $\bar\alpha\ge 0$ such that
  \begin{equation*}
  ((\cD_0+\bar\alpha\cK)W,W)\ge c_0\|W\|^2,~~~W\in \bar H^{1/2};
\end{equation*}
  \item[(ii)] there exists $\eta <1$ and $\delta,C>0$ such that\footnote{The relation in \eqref{Pi-0-es} is the standard requirement in many second order in time models, see, e.g., \cite{cl-mem}. } 
\begin{equation}\label{Pi-0-es}
 \delta \bar \Pi_0(U) -(\cB (U),U)\le  \eta \|\cA^{1/2}U\|^2 +C \;,\quad U\in \bar H^{1/2},
\end{equation}
where $\bar \Pi_0(U)= \Pi_{01}(u)+ \Pi_{02}(v)$
with $U=(u;v)$.   
\end{enumerate}	
\end{assumption}

\begin{theorem}\label{th:-u-dis}
Let Assumptions \ref{A1-sync2} and \ref{as:u-diss} be in force. Then 
for every $\alpha\ge\bar \alpha$ and $\vka\ge 0$
the system 
$(\cH,S_t)$ generated by problem \eqref{ver-problem-kappa} is dissipative\footnote{See the Appendix for the notions related with this theorem.}
 with an absorbing ball  of the size  independent
of  $(\alpha;\vka) \in     \La\equiv \{ \alpha\ge\bar \alpha,~\vka\ge 0\}$. More precisely,	there exists $R$ independent of  
$(\alpha;\vka) \in\La$ such that the set
\begin{equation}\label{set-B-uni}
  \fB=
\left\{ (U_0;U_1)\in \cH\, : \|U_1\|^2+  \|\cA^{1/2} U_0\|^2
+\vka \|\cK^{1/2}U_0\|^2\le R^2\right\}
\end{equation}
is absorbing.
\end{theorem}
 We note that the estimate in \eqref{set-B-uni} improves the corresponding finite-dimensional statement 
 in \cite{ACH-97} which requires uniform boundedness of the ratio $\al/\vka$. As it is shown in Proposition~\ref{pr:qs-U}
 we need the latter property for uniform quasi-stability only.

\begin{proof} We use a slight modification of the standard method
(see, e.g., \cite{BV92,Chu99,temam})
 based on Lyapunov type functions.
Let $U(t)=(u(t);v(t))$ be a strong solution,
\[
E(t)\equiv E(U;U_t)=\hf \left(\|U_t(t)\|^2+  \|\cA^{1/2} U(t)\|^2\right)+\bar\Pi(U(t))
\]
with $\bar\Pi(U)=\Pi_1(u)+\Pi_2(v)$,
and $\Phi(t)=\eta (U,U_t)+\mu (\cK U,U)$,
where $\eta$  is a positive constant which will be chosen later and $2\mu=\vka+\eta \al$.
We consider 	the functional $V=E+\Phi$.
One can see that there exist $0<\eta_0<1$ and $\beta_i>0$ independent of  $(\alpha;\vka)$
such that 
\begin{equation}\label{V-eqv-norm}
 \beta_0\big[ E_*(t)+\vka \|\cK^{1/2} U(t)\|^2\big]-\beta_1 \le V\le 
 \beta_2\big[ E_*(t)+\mu  \|\cK^{1/2} U(t)\|^2\big]+\beta_3
\end{equation}
for all $\eta\in (0,\eta_0]$, where 
\[
E_*(t)\equiv E_*(U;U_t)=\hf \left(\|U_t(t)\|^2+  \|\cA^{1/2} U(t)\|^2\right)+\bar\Pi_0(U).
\]
Now on strong solutions using the energy relation in \eqref{8.1.4} we calculate the derivative
\begin{align*}
\frac{dV}{dt}=  -((\cD_0+\alpha \cK)U_t,U_t)	
+\eta\big[ \|U_t\|^2 -(\cD_0U_t,U)-(\cA U,U) -\vka (\cK U,U)-(\cB(U),U)\big].
\end{align*}
Since $\cD_0$ is bounded from $\bar H^{1/2}$ into $\bar H$, we obtain that
\[ 
|(\cD_0U_t,U)|\le \eps\|\cA^{1/2} U\|^2 +C\eps^{-1}\|U_t||^2,~~~\forall  \eps>0.
\]
Thus
by Assumption~\ref{as:u-diss}(ii) 
there exist $b_i>0$ independent
of $(\alpha,\vka)$ such that
\begin{align*}
\frac{dV}{dt}\le   -\big[((\cD_0+\alpha \cK)U_t,U_t)
-b_1\eta\|U_t\|^2\big]	
-b_2\eta\big[E_*(t) +\vka (\cK U,U)\big] +\eta b_3,
\end{align*}
 This implies that there exists $0<\eta_*\le \eta_0$ independent of 
  $(\alpha,\vka)\in\Lambda$ such that
 \begin{equation}\label{v-beta-bnd}
\frac{dV_\beta}{dt}+
b_2\eta\big[E_*(t) +\vka (\cK U,U)\big] \le \eta b_3
\end{equation}
 for all $(\alpha,\vka)\in\Lambda$ and $\eta\in (0,\eta_*]$, where
$V_\beta=V+\beta_1>0$.

Now we split the parametric region $\La$ into
several subdomains. 

We start with the following case.
Let $\vka_*>0$ and $\alpha_*>\bar\al$ be fixed.
 We take $0<\tilde \vka\le \vka_*$ such that
 $\al_*>\tilde \vka/\eta_*$ Then  we take
 $\eta=\tilde\vka\al^{-1}$.  
 In this case $\eta \le\eta_*$ and
 \[
 \mu=\hf (\vka +\eta \al)=\hf (\vka + \tilde\vka)\le \vka.
 \]
 Thus for all $\vka\ge \vka_*$ and $\al\ge \al_*$ we have that
 \[
\frac{dV_\beta}{dt}+	
b_2\eta\big[E_*(t) +\mu (\cK U,U)\big] \le \eta b_3.
\]  
 In particular, \eqref{V-eqv-norm} yields
 \begin{equation*}
  \frac{dV_\beta}{dt}+
b_2 \eta \beta_2^{-1}  V_\beta\le \eta b_4~~\mbox{with}~~ b_4=b_3+ \frac{\beta_1+\beta_3}{\beta_2},
\end{equation*}
where $b_i$ and $\beta_i$ do not depend on $\al$ and $\vka$.
This implies that
\begin{equation*}
V_\beta (t)\le V_\beta (0)e^{-b_2\eta \beta_2^{-1} t}
+\frac{b_4\beta_2}{b_2}. 
\end{equation*}
Since the value 
$b_4 b_2^{-1}\beta_2$
is independent of
$\vka\ge 0$ and $\al\ge \bar\al$, 
we can conclude from the previous argument that
 the set $\fB$ is absorbing for all
 $\vka>0$ and $\al> \bar\al$ 
 with $R$ independent of $\vka$ and $\al$.
\par 
In the case when $\vka=0$ 
from \eqref{v-beta-bnd}
we obtain
\[
\frac{dV_\beta}{dt}+
b_2\eta E_*(t)  \le \eta b_3.
\]
Now we take $\al\ge \eta_*^{-1}$ and $\eta=\al^{-1}$. In this case $\eta\le \eta_*$ and $\mu=1/2$. Therefore the conclusion follows from  the estimate 
$\|\cK^{1/2} U\|\le c\|\cA^{1/2} U\|$ by the same argument as above. 

In the case $\vka=0$ and 
$\bar\al\le\al\le  \eta_*^{-1}$ the conclusion is obvious.
\par 
So the remaining case is $\vka> 0$ and $\al=\bar\al$. Now we can take 
$\eta=\min\{\eta_*, \vka \bar\al^{-1}\}$ 
when $\bar\al>0$. It is clear that $\mu \le \vka$ for this case. Thus we can argue as above.
 In the case $\bar\al=0$ the relation 
 $\mu\le\vka/2$ holds automatically.
This completes the proof Theorem \ref{th:-u-dis}. 
\end{proof}

This theorem and Proposition~\ref{th:-a-comp}
immediately imply  the following result on the existence of a global attractor.
 
\begin{theorem}\label{th:-u-attr}
Let Assumptions \ref{A1-sync2}, 
\ref{as:a-comp}(ii) and \ref{as:u-diss} be in force. Then for every $(\alpha;\vka) \in\La$ 
the system 
$(\cH,S_t)$ generated by problem \eqref{ver-problem-kappa} 
possesses a compact global attractor $\fA$.
For every full trajectory $Y=\{(U(t);U_t(t))\, :
t\in\R\}$ from the attractor 
\begin{equation}\label{bound-attr}
\sup_{t\in\R} \left\{\|U_t(t)\|^2+  \|\cA^{1/2} U(t)\|^2
+\vka \|\cK^{1/2}U(t)\|^2\right\} +
\int_{-\infty}^{\infty}(\cD_0+\al\cK)U_t(\t), U_t(\t))\ d\t\le R^2
\end{equation}
for some $R$ independent of  $(\alpha;\vka) \in\La$.
\end{theorem}
\begin{proof}
We first apply the standard result on the existence of a global attractor, see Theorem \ref{th:main-attractor-a} in the Appendix. This attractor belongs to the set $\fB$ defined in \eqref{set-B-uni}. 
This implies an uniform bound for the supremum in \eqref{bound-attr}. Using the energy relation in \eqref{8.1.4new} we obtain the corresponding bound for
the dissipation integral in \eqref{bound-attr}. 
\end{proof}

\bigskip

%
%
%
%
%

\subsection{Quasi-stability}
The uniform bounds for the attractor given by
Theorem \ref{th:-u-attr}
are not sufficient to perform the large coupling limit $\al\to+\infty$ and/or $\vka\to+\infty$ in the phase state of the system.
One of  achievements  of this section is stronger uniform estimates for the attractor size.
For this we apply the quasi-stability 
method in the form suggested in
\cite{cl-mem}
(see also \cite{C-Eller-L-2004,cl-jdde,jde,jde07,cl-book} and the survey in the recent monograph \cite{Chu-dqsds15}).
This method makes it also possible to proof finite-dimensionality of the attractor and obtain its smoothness
properties.

To apply the quasi-stability method
we need additional hypotheses concerning the nonlinear forces $B_i(u)$.
\begin{assumption}\label{as-sm-tr}{\rm
Assume that
\begin{itemize}
  \item  $B_i(u)=\Pi_i'(u)$ with the functional $\Pi_i\, :\, H^{1/2} \mapsto \R$
  which  is a Fr\'echet $C^3$-mapping.
  \item The second  $\Pi_i^{(2)}(u)$ and the third $\Pi_i^{(3)}(u)$ Fr\'echet
  derivatives of $\Pi_i(u)$ satisfy the conditions
\begin{equation}\label{pi-2-bnd}
\left|\langle \Pi_i^{(2)}(u); v, v\rangle\right|\le C_\rho \|A^{\sigma}v\|^2,
\quad v\in H^{1/2},
\end{equation}
for some $\sigma<1/2$, and
\begin{equation}\label{pi-3-bnd}
\left|\langle \Pi_i^{(3)}(u); v_1, v_2, v_3\rangle\right|\le
C_\rho \|A^{1/2}v_1\| \|A^{1/2}v_2\| \|v_3\| ,\quad v_i\in H^{1/2},
\end{equation}
for all  $u\in H^{1/2}$ such that $\|A^{1/2}u\|\le \rho$,
where $\rho>0$ is arbitrary and $C_\rho$ is a positive constant.
Here above $\langle \Pi_i^{(k)}(u) ; v_1,\ldots,v_k\rangle$
denotes the value of the derivative  $\Pi_i^{(k)}(u)$ on elements $v_1,\ldots,v_k$.
\end{itemize}
}
\end{assumption}
This assumption concerning nonlinear feedback forces $B_i(u)$  appeared earlier
in the case of  systems with nonlinear
damping (see \cite[p. 98]{cl-mem} and also \cite{kolbasin}) to cover
the case of critical nonlinearities.
We  note that Assumption \ref{as-sm-tr} holds in both cases of the von Karman and Berger models (see \cite{cl-mem} p.\ 156 and p.\ 160 respectively). Moreover, as it is shown in \cite[p.137]{cl-mem}, this Assumption~\ref{as-sm-tr} is also true in the case of the coupled  3D wave equation in a bounded domain $\Omega\subset \R^3$ of the form
\begin{subequations}\label{3d-wave}
\begin{align}
& u_{tt}+\s_1 u_t -\Delta u+ k_1(u-v) +\f_1(u)=f_1(x), \quad u\big|_{\partial\Omega}=0, \\
& v_{tt}+\s_2 v_t -\Delta v+ k_2(u-v) +\f_2(u)=f_2(x), \quad  v\big|_{\partial\Omega}=0,
\end{align}
\end{subequations}
provided $\f_i\in C^2(\R)$ possesses the property $|\f_i''(s)|\le C(1+|s|)$
for all $s\in\R$, the parameters $\sigma_i$ and $k_i$ are nonnegative.
Thus our abstract  model covers the case of 3D wave dynamics 
with a {\em critical} force term. We refer to Section \ref{sec:cpl-wave}
for a further discussion concerning nonlinear wave equations.
\par
Recall that  the Fr\'echet derivatives  $\Pi^{(k)}(u)$  of the functional
$\Pi$ are {\it symmetric} $k$-linear continuous forms on $H^{1/2}$
 (see, e.g., \cite{cartan}).
Moreover, if $\Pi\in C^3$, then $(B(u),v)\equiv\langle \Pi'(u); v\rangle$
is $C^2$-functional for every fixed  $v\in H^{1/2}$ and
the following Taylor's expansion holds
\begin{equation}\label{taylor}
(B(u+w)- F(u), v)=
\langle \Pi^{(2)}(u); w, v\rangle
+\int_0^1(1-\lambda)
\langle \Pi^{(3)}(u+\lambda w);w, w, v\rangle d\lambda
 \end{equation}
for any $u,v\in H^{1/2}$ \cite{cartan}.   If
Assume that $u(t)$ and $z(t)$ belong
 to the class $C^1\big(a,b; \sD(\cA^{1/2})\big)$ for some interval $[a,b]\subseteq \R$.
Then, by the differentiation rule for composition of mappings \cite{cartan} and using the symmetry of the form $\Pi^{(2)}(u)$, we have that
\[
\frac{d}{dt}\langle \Pi^{(2)}(u); z, z\rangle=
\langle \Pi^{(3)}(u);u_t, z, z\rangle +2\langle \Pi^{(2)}(u); z, z_t\rangle.
\]
Therefore from \eqref{taylor} we obtain the following representation which is important in our further considerations:
\begin{equation}\label{newsm}
(B_i(u(t)+z(t))  - B_i(u(t)), z_t(t) )  = \frac{d}{dt} Q_i(t) +R_i(t),\quad t\in [a,b]\subseteq \R,
\end{equation}
with
\begin{equation}\label{q-new}
  Q_i(t)=-\frac12\langle \Pi_i^{(2)}(u(t)); z(t), z(t)\rangle
\end{equation}
and
\begin{equation}\label{r-new}
 R_i(t)= -\frac{1}{2} \langle \Pi_i^{(3)}(u);u_t, z, z\rangle+
 \int_0^1(1-\lambda)\langle \Pi_i^{(3)}(u+\lambda z); z, z ,z_t\rangle d\lambda.
\end{equation}
As we will see below the representation in \eqref{newsm} and 
  the hypotheses
listed in Assumption~\ref{as-sm-tr}
can be avoided if 
 the nonlinear forces $B_i(u)$ are subcritical, i.e.,
\begin{equation}\label{8.1.1c-sub}
\exists\, \sigma_0<1/2:~~\|B_i(u_1) -B_i(u_2)|~\leq~L(\varrho) \|\cA^{\sigma_0}( u_1-u_2)\|,\quad
\forall \|\cA^{1/2} u_i \|  \leq \varrho.
\end{equation}

The representation in \eqref{newsm} leads to
 the following assertion which, in fact, is proved in \cite{cl-mem} (see (4.38), p.\ 99), but without control of the parameters $\al$ and $\vka$.

\begin{proposition}\label{pr:qs-U}
Let Assumptions \ref{A1-sync2} and \ref{as:u-diss} be in force.
In addition 
we suppose that either Assumption \ref{as-sm-tr} or relation (\ref{8.1.1c-sub})
holds.
Let $\cM\subset \cH$ be a bounded forward invariant set \wrt 
$S_t$ and 
 $U^i=S_tU^i_0,\ i=1,2$ be two solutions   to (\ref{ver-problem-kappa}) 
with   (different) initial data $U_0^i\in \cM$. 
Let $Z=U^1-U^2$. Then there exist $C, \ga>0$ such that
\begin{equation} \label{stab-est-mu}
    E_Z(t) \leq C E_Z(0) e^{-\ga t} +
    C \underset{[0,t]}{\max}\|\cA^{\sigma}Z(\t)\|^2, \quad \forall t>0,
\end{equation}
where $0\le\sigma<1/2$ and  
\[
E_Z(t)=\frac12\left(\|Z_t(t)\|^2+ \|\cA^{1/2}Z(t)\|^2 +\vka \|\cK^{1/2} Z\|^2\right).
\]
If $\cM$ is a uniformly  bounded in $\cH$ \wrt $(\al;\vka)\in \Lambda$ and
$$(\al;\vka)\in \Lambda_\beta\equiv
\{ (\al;\vka)\in \La \, :\, \alpha\le \beta(1+\vka)\}$$ for some $\beta>0$, then the constants $C,\gamma$ independent of $(\al;\vka)$
but can depend on $\beta$.
\end{proposition}
\begin{proof}
We use the same line of argument as in 
\cite{cl-mem} and start with the following relation (which follows from Lemma~3.23 in \cite{cl-mem}):
\begin{align*}
 TE_Z(T) +
\int_{0}^{T} E_Z(t) dt
\le & c\left\{\int_0^T ((D_0+\al\cK) Z_t,Z_t) dt\right.
  \\
 & + \left.  \int_0^T \left| ((D_0+\al\cK) Z_t,Z)\right| dt+\Psi_T(U^1,U^2)\right\}
\end{align*}
for  every $T\ge T_0\ge 1$,
where $c>0$ does not depend on $\al,\vka,T$ and
\begin{align*}
\Psi_T(U^1,U^2) = & \left|  \int_{0}^{T}
(G(\tau), Z_t(\tau) ) d\tau\right|
+  \left| \int_{0}^{T}(G(t), Z(t) ) dt\right|   +   \left|  \int_{0}^{T}dt \int_{t}^{T}
(G(\tau), Z_t(\tau) ) d\tau \right|
\end{align*}
with  
\begin{equation*}
G(t)= \cB ( U^1(t)) - \cB ( U^2(t)).
\end{equation*}
Since every point $(\al;\vka)\in \Lambda$
belongs to $\Lambda_\beta$ for some $\beta>0$.
it is sufficient to consider the case when 
$(\al;\vka)\in\Lambda_\beta$ for some $\beta$. In this case we have that
\begin{align*}
\left| ((D_0+\al\cK) Z_t,Z)\right|\le & 
 ((D_0+\al\cK) Z_t,Z_t)^{1/2} ((D_0+\al\cK) Z,Z)^{1/2} \\ 
 \le & C_{\eps,\beta}((D_0+\al\cK) Z_t,Z_t) +\eps E_Z(t)
\end{align*}
for every $\eps>0$. Thus choosing $\eps$ in an appropriate way we obtain that
\[
 TE_Z(T) +
\int_{0}^{T} E_Z(t) dt
\le  c_\beta\int_0^T ((D_0+\al\cK) Z_t,Z_t) dt+ c_0\Psi_T(U^1,U^2).
\]
Under Assumption \ref{as-sm-tr} we have from Proposition 4.13 in \cite{cl-mem}
that
for any $\eps>0$ and $T>0$ there exist $a(\eps,T)=a_\cM(\eps,T)$  
and $b(\eps,T)=b_\cM(\eps,T)$
 such that 
\begin{align}\label{Gzt-new}
\sup_{t\in[0,T]}
\left|  \int_{t}^{T}
(G(\tau), Z_t(\tau) ) d\tau \right|
\le  &\eps \int_0^{T}\left[ \|Z_t(\t)\|^2+  \|\cA^{1/2} Z(\tau)\|^2\right] d\tau 
\\
&  + a(\eps,T) \int_0^{T}\ d(\tau)\|\cA^{1/2} Z(\tau)\|^2 d\tau
+\; b(\eps,T) \sup_{\t \in [0,T]} \|\cA^{\sigma} Z(\tau)\|^2\nonumber
\end{align}
for all  $\e>0$, where $\sigma<1/2$ and
\begin{equation*}
d(t)\equiv d(t; U_1, U_2)=\| U^1_t(t)\|^2+\| U_t^2(t)\|^2.
\end{equation*}
Obviously the same relation \eqref{Gzt-new} (even with $a(\eps,T)\equiv 0$) remains true in the subcritical case \eqref{8.1.1c-sub}. Thus
\begin{equation*}
\Psi_T(U^1,U^2)
\le  \eps \int_{0}^{T} E_Z(\t) d\t 
  + a(\eps,T) \int_0^{T}\ d(\tau)\|\cA^{1/2} Z(\tau)\|^2 d\tau
+ b(\eps,T) \sup_{\t \in [0,T]} \|\cA^{\sigma} Z(\tau)\|^2
\end{equation*}
for every $\eps>0$.
From the energy relation we also obtain that
\[
\int_0^T ((D_0+\al\cK) Z_t,Z_t) dt\le
E_Z(0)-E_Z(T)+\Psi_T(U^1,U^2)
\]
Thus after appropriate choice of $\eps$
and $T$ we arrive at the relation
\begin{equation*}
E_Z(T)\le q E_Z(0)+	 a(T) \int_0^{T}\ d(\tau)\|\cA^{1/2} Z(\tau)\|^2 d\tau
+ b(T) \sup_{\t \in [0,T]} \|\cA^{\sigma} Z(\tau)\|^2,
\end{equation*}
where $q<1$ and all constants depend on $\beta$.
This inequality allows us to apply the same procedure as in \cite[p.100]{cl-mem}
to obtain \eqref{stab-est-mu}.
\end{proof}
\medskip\par 

Now we are in position to obtain a result on the finiteness
of fractal dimension of the attractors and also additional bounds for trajectories from these attractors.

\begin{theorem}\label{th:attr2}
Let Assumptions \ref{A1-sync2}
 and \ref{as:u-diss} be in force.
In addition 
we suppose that either Assumption \ref{as-sm-tr} or relation (\ref{8.1.1c-sub})
holds. Then for any  $(\al;\vka)\in \Lambda$ the following assertions hold:
\begin{enumerate}
 \item[{\bf 1.}]  The global attractor $\fA^{\al,\vka}$ of the system
  $(\cH, S_t)$ generated by \eqref{ver-problem-kappa}  has a finite fractal dimension  ${\rm dim}_f\fA$.
  \item[{\bf 2.}] This attractor $\fA^{\al,\vka}$ lies in $\bar H^1\times \bar H^{1/2}$ and for
  every full trajectory $Y=\{(U(t);U_t(t))\, :
t\in\R\}$ from the attractor in addition to
the bound in \eqref{bound-attr} we have that
\begin{equation}\label{bound-attr+}
\sup_{t\in\R} \left\{\|U_{tt}(t)\|^2+  \|\cA^{1/2} U_t(t)\|^2
+\vka \|\cK^{1/2}U_t(t)\|^2 \right\}\le R_1(\beta)
\end{equation}
for some $R_1(\beta)$ independent of  $(\alpha;\vka) \in\La_\beta$, where   $\beta>0$ can be arbitrary.
 \item[{\bf 3.}] The attractors  $\fA^{\al,\vka}$ are upper semicontinuous at every point  $(\al_*;\vka_*)\in \La$,
 i.e., 
\begin{equation}\label{u-u-upper}
\lim_{n\to \infty}\left[\sup\big\{ {\rm dist}_{\cH}(y,\fA^{\al_*,\vka_*}):y\in\fA^{\al^n,\vka^n} \big\}\right] = 0
\end{equation}
 for every sequence $\{(\al^n;\vka^n)\}\subset  \La$ such that $(\al^n;\vka^n)\to(\al_*;\vka_*)\in \La$ as $n\to\infty$. 
\end{enumerate}
\end{theorem}
\begin{proof}
By Proposition \ref{pr:qs-U} the system 
$(\cH, S_t)$ is quasi-stable on  every bounded forward invariant set. Thus we can apply Theorems 3.4.18 and 3.4.19 from
\cite{Chu-dqsds15} to prove the statements 1 and 2 (see also Theorems \ref{th7.9dim} and \ref{th7.9reg} in the Appendix). 
\par 
To prove upper semicontinuity property we can use the methods developed in \cite{Hale-Rauge-upper,kap-kos}. In particular, we can apply a result due to \cite{kap-kos}
(see Theorem~\ref{t7.2.5} in the Appendix).
Indeed,  let 
$(\al^n;\vka^n)\to(\al_*;\vka_*)\in \La$ as $n\to\infty$.
In this case
it follows from \eqref{bound-attr+} that
\[
\|\cA U(t)\|^2+  \|\cA^{1/2} U_t(t)\|^2
 \le C^2,~~\forall\, t\in\R, 
\]
for every full trajectory $(U(t);U_t(t))$ from the attractor $\fA^{\al_n,\vka_n}$, where $C$ does not depend on $n$.
This means that the attractor $\fA^{\al_n,\vka_n}$ belongs to the set
\[
\big\{ (U_0,U_1)\, : \; \|\cA U_0\|^2+  \|\cA^{1/2} U_1\|^2
 \le C^2\big\}
 \]
which is compact in $\cH$. Thus we only need to show the property (ii) in Theorem~\ref{t7.2.5}.
\par
Let $(U_0^n;U_1^n)\in\fA^{\al_n,\vka_n}$ and $(U_0;U_1)\in\fA^{\al_*,\vka_*}$.
One can see that
\[
(Z(t);Z_t(t))=S_t^{\al_n,\vka_n}(U_0^n;U_1^n)-S_t^{\al_*,\vka_*}(U_0;U_1)\equiv (U^n(t)-U(t);U_t^n(t)-U_t(t))
\]
satisfies the equation 
\[
 Z_{tt} +\cA Z+(\cD_0+ \alpha_* {\cal K}) Z_t+ \vka_* \cK Z =F,
 \]
 where
 \[
 F = - (\alpha_n-\al_*) {\cal K} U^n_t-
 (\vka_n-\vka_*) \cK U^n  -{\cal B}(U^n)
  +{\cal B}(U).
\]
On the attractors we obviously have that
 \[
\| F\| \le c_1 (|\alpha_n-\al_*|+|
 \vka_n-\vka_*|) +c_2\|\cA^{1/2} Z\|.
\]
Therefore the standard energy type calculations gives the estimate
\[
\|S_t^{\al_n,\vka_n}Y_0^n-S_t^{\al_*,\vka_*}Y_0\|_{\cH}
\le C (|\alpha_n-\al_*|+|
 \vka_n-\vka_*|+ 
 \|Y_0^n-Y_0\|_{\cH}) e^{a t},~~ t>0, 
\]
where $Y^n_0=(U_0^n;U_1^n)$ and 
$Y_0=(U_0;U_1)$. Thus we can apply Theorem~\ref{t7.2.5}.
\end{proof}
\begin{remark}\label{re:N-eqs}
{\rm 
	The results stated in Theorems \ref{t2.wp-smp} and \ref{th:attr2} deal with a general model of the form \eqref{ver-problem} or (\ref{ver-problem-kappa}) and thus they can be also applied
	in the case of several interacting second order in time equations of the form
		\begin{equation}\label{eq-N}
    u^i_{tt} +\nu_i A u^i+  D_{0i}u^i_t +\al\sum_{j=1}^N K_{ij}u^j_t
    + \vka\sum_{j=1}^N K_{ij}u^j +B_i(u^i)=0,
 ~~i=1,\ldots,N,
 \end{equation}
under obvious changes in the set of hypotheses concerning the operators in \eqref{eq-N}.
}
\end{remark}

\subsection{Asymptotic synchronization}
Now we
apply the results above to synchronization.
We  switch on the interaction operators
$\cK$ of the standard (see, e.g., 
\cite{CRD,Hale-jdde97,naboka07,H} and the references therein) symmetric form. 
Moreover we suppose that the damping operator $\cD_0$ has a diagonal structure.
Namely we assume that
\begin{equation}\label{k-d-diag}
  \cK=\left(\begin{matrix}
 1 & -1 \\
 -1 & 1 
\end{matrix}
\right) K, ~~ \cD_0=\left(\begin{matrix}
 D_1 & 0 \\
 0 & D_2 
\end{matrix}
\right) 
\end{equation}
where $K$ is a strictly positive operator in $H$ with domain $\cD(K)\supseteq  H^{1/2}$, the operators $D_i : H^{1/2}\mapsto H$ are nonnegative.
Thus we consider the following problem
\begin{subequations}\label{wave-eq1s}
 \begin{equation}\label{wave-eq1as}
    u_{tt} +\nu_1 A u+ D_{1}u_t +\al K(u_t-v_t)
    +\vka  K(u-v) +B_1(u)=0,
 \end{equation}
 \begin{equation}\label{wave-eq1bs}
    v_{tt} +\nu_2 A v+ D_{2}v_t +\al K(v_t-u_t)
    +\vka K(v-u) +B_2(v)=0,
 \end{equation}
 with initial data
 \begin{equation}
 u(0)=u_0,~~u_t(0)=u_1,~~  v(0)=v_0,~~v_t(0)=v_1.
\end{equation}
 \end{subequations}
All theorems stated above can be applied to this situation. 

Our goal is to study 
asymptotic synchronization phenomena and we are interested in qualitative  
behavior of the system in the large coupling 
limit  $\vka\to\infty$ (and/or $\al\to\infty$).
It is clear from the bound of the attractor given in \eqref{bound-attr} that it is reasonable to assume that $u=v$ in the this limit.
Therefore we need to consider a limiting problem of the form
\begin{equation}\label{wave-eq1as-lim}
    w_{tt} +\nu A w+ Dw_t +B(w)=0,
    ~~  w(0)=w_0,~~w_t(0)=w_1.
 \end{equation}
 where
 \begin{equation*}
   \nu=\frac{\nu_1 +\nu_2}2,~~
   D=\hf (D_1+ D_{2}),~~ B(w)=\hf(B_1(w) +B_2(w)).
 \end{equation*}
Obviously the argument above can be applied to system \eqref{wave-eq1as-lim} provided the damping operator $D$ is not degenerate. 
In fact we can easily prove the following assertion.

\begin{proposition}\label{pr:lim-syst}	
Let Assumptions \ref{A1-sync2}(i,iv), 
 and \ref{as:u-diss}(ii) be in force.
Let $D : H^{1/2}\mapsto H$ be a strictly positive operator.
In addition 
assume  that either Assumption \ref{as-sm-tr} or relation (\ref{8.1.1c-sub}) is valid.
Then problem \eqref{wave-eq1as-lim} generates 
a dynamical system in the space $H^{1/2}\times H$  possessing a compact global attractor of finite fractal dimension. This attractor is a bounded set in $H^{1}\times H^{1/2}$. 
\end{proposition}
Below
we show that the
attractors  ${\fA^{\al,\vka}}$ for problems (\ref{wave-eq1s}) in some sense converge to the
attractor  of the limiting system
\eqref{wave-eq1as-lim} when $\vka\to+\infty$.

\begin{theorem}\label{th:att-sync}
Let Assumptions \ref{A1-sync2}, 
\ref{as:a-comp}(ii) and \ref{as:u-diss} be in force with $\cK$ and $\cD_0$ of the form given in \eqref{k-d-diag}.
Then for every $(\alpha;\vka) \in\La$ 
the system 
$(\cH,S_t)$ generated by problem \eqref{wave-eq1s} 
possesses a compact global attractor $\fA^{\al,\vka}$.
For every full trajectory $Y=\{(U(t);U_t(t))\, :
t\in\R\}$ with $U(t)=(u(t);v(t))$ from the attractor 
\begin{equation}\label{bound-attr-sy}
\sup_{t\in\R} \left\{\|U_t(t)\|^2+  \|\cA^{1/2} U(t)\|^2
+\vka \|K^{1/2}(u(t)-v(t))\|^2 \right\}+\al\int_{-\infty}^{+\infty} \|U_t(\t)\|^2 d\t\le R^2
\end{equation}
for some $R$ independent of  $(\alpha;\vka) \in\La$.
\par 
In addition 
assume  that either Assumption \ref{as-sm-tr} or relation (\ref{8.1.1c-sub})
holds. Then for any  $(\al;\vka)\in \Lambda$ the following assertions hold:
\begin{enumerate}
 \item[{\bf 1.}]  The global attractor $\fA^{\al,\vka}$ of the system
  $(\cH, S_t)$ generated by \eqref{wave-eq1s}  has a finite fractal dimension  ${\rm dim}_f\fA$.
  \item[{\bf 2.}] This attractor $\fA^{\al,\vka}$ lies in $\bar H^1\times \bar H^{1/2}$ and for
  every full trajectory $Y=\{(U(t);U_t(t))\, :
t\in\R\}$ from the attractor in addition to
the bound in \eqref{bound-attr-sy} we also have that
\begin{equation}\label{bound-attr+sy}
\sup_{t\in\R} \left\{\|U_{tt}(t)\|^2+  \|\cA^{1/2} U_t(t)\|^2
+\vka \|K^{1/2}(u_t(t)-v_t(t))\|^2 \right\}\le R^2_1(\beta)
\end{equation}
for some $R_1(\beta)$ independent of  $(\alpha;\vka) \in\La_\beta$, where   $\beta>0$ can be arbitrary.
 \item[{\bf 3.}] The attractors  $\fA^{\al,\vka}$ are upper semicontinuous at every point  $(\al_*;\vka_*)\in \La$,
 i.e., \eqref{u-u-upper} is valid
 for every sequence $\{(\al^n;\vka^n)\}\subset  \La$ such that $(\al^n;\vka^n)\to(\al_*;\vka_*)\in \La$ as $n\to\infty$. 
 \item[{\bf 4.}] 
 Let $u\mapsto B_i(u)$ be weakly continuous from $H^{1/2}$ into some space $H^{-l}$, $l\ge 0$. Then
 in the limit $\vka\to\infty$ we have that
\begin{equation}\label{u-u-upper-inf}
\lim_{\vka \to \infty}\left[\sup\big\{ {\rm dist}_{\cH_\eps}(Y,\widetilde{\fA}):Y\in\fA^{\al,\vka} \big\}\right] = 0,
\end{equation}
where $\cH_\eps=\bar{H}^{1/2-\eps}\times 
 \bar{H}^{1/2-\eps}$ and 
 $\widetilde{\fA}=\big\{ (u_0;u_0;u_1;u_1): 
 (u_0;u_1)\in\fA\big\}$. Here 
$\fA$ is the global attractor
for the dynamical system  generated 
by \eqref{wave-eq1as-lim}.
Moreover, if instead of weak continuity of  $B_i$ we assume that $\nu_1=\nu_2$ and 
$K$ commutes\footnote{We can take $K=A^\sigma$ with some $0\le\sigma\le1/2$, for instance.} with $A$, then the convergence
in \eqref{u-u-upper-inf} holds in the space
$\bar{H}^{1/2}\times 
 \bar{H}^{1/2-\eps}\subset \cH$.
\end{enumerate}
\end{theorem}
\begin{proof}
All results except  the last one easily follows from Theorem~\ref{th:attr2}. 
Thus we need to establish property \eqref{u-u-upper-inf} only. As in \cite{Hale-Rauge-upper,kap-kos} we apply contradiction argument.
\par Assume that \eqref{u-u-upper-inf} is not true. Then there exist  sequences 
$\{\vka_n\to\infty\}$ and $Y^n_0\in \fA^{\al,\vka_n}$ such that
\[
{\rm dist}_{\cH_\eps}(Y^n_0,\widetilde{\fA})\ge \delta>  0, ~~ n=1,2,\ldots 
\]
Since $Y^n_0\in \fA^{\al,\vka_n}$, there exists a full trajectory 	
 $Y^n=\{(U^n(t);U^n_t(t))\, :
t\in\R\}$ from the attractor $\fA^{\al,\vka_n}$ such that $Y^n(0)=Y^n_0$.
It follows from \eqref{bound-attr-sy} and
\eqref{bound-attr+sy} and also from Aubin-Dubinsky-Lions theorem (see \cite{sim}, Corollary 4)  that the sequence $Y^n$ is compact in $C(a,b;\cH_\eps)$ for every $a<b$ and $*$-weakly compact in $L_\infty(\R; \bar{H}^{1/2}\times\bar{H}^{1/2})$.Thus there exists 
\[
\widetilde{U}(t)=(u(t);v(t))\in 
C^1(\R;\cH_\eps)\cap L_\infty(\R; \bar{H}^{1/2}\times\bar{H}^{1/2})~~\mbox{with}~~
\widetilde{U}_t\in  L_\infty(\R; \bar{H}^{1/2}\times\bar{H}^{1/2})
\]
such that along a subsequence 
\begin{equation}\label{n-conv-1}
 \forall\, a<b:~~ \sup_{t\in [a,b]} \left\{\|U^n_t(t)-\widetilde{U}_t(t) \|_{1/2-\eps}^2+  \|U^n(t)-\widetilde{U}(t) \|_{1/2-\eps}^2
 \right\}\to 0,~~ n\to\infty,
\end{equation}
and
\begin{equation*}
  (U^n(t);U^n_t(t))\to (\widetilde{U}(t);\widetilde{U}_t(t))~~\mbox{$*$-weakly in}~~
   L_\infty(\R; \bar{H}^{1/2}\times\bar{H}^{1/2}),~~ n\to\infty.
\end{equation*}
Since $K$ is strictly positive, it follows from  \eqref{bound-attr-sy} and
\eqref{bound-attr+sy}  that
\[
\sup_{t\in\R} \left\{ \|u^n_t(t)-v^n_t(t)\|^2  +
\|u^n(t)-v^n(t)\|^2 \right\}\le \frac{R^2}{\vka_n}\to 0,~~ n\to\infty.
\]
By interpolation
\begin{align*}
	\sup_{t\in\R}  \| A^{1/2-\eps}(u^n_t(t)-v^n_t(t))\|  & \le
	C \sup_{t\in\R} \left\{ \left(\|A^{1/2}u^n_t(t)\|+\|A^{1/2}v^n_t(t)\|\right)^{1-2\eps} 
	 \|u^n_t(t)-v^n_t(t)\|^{2\eps} \right\}
	 \\ &
	 \le  C \sup_{t\in\R}  
	 \|u^n_t(t)-v^n_t(t)\|^{2\eps} 
	 \to 0,~~ n\to\infty.
\end{align*}
Similarly 
\[
\sup_{t\in\R}  \| A^{1/2-\eps}(u^n(t)-v^n(t))\| 
	 \to 0,~~ n\to\infty.
	 \]
Since $u\mapsto B_i(u)$ is weakly continuous for $i=1,2$,	 
these observations allow us to make a limit transition in the variational form of equations \eqref{wave-eq1s} and conclude that $\widetilde{U}(t)=(u(t);u(t))$, where $u(t)$ is a solution to \eqref{wave-eq1as-lim}. Moreover, $(u(t);u_t(t))$ is a trajectory bounded in $H^{1/2}\times H^{1/2}$. 
Thus it belongs to the attractor $\fA$.
It is also clear from \eqref{n-conv-1} that
\[
Y_0^n=  (U^n(0);U^n_t(0))\to (\widetilde{U}(0);\widetilde{U}_t(0))\in\widetilde{\fA}
~~\mbox{in}~~ \cH_{\eps}
\]
which is impossible.
\medskip\par 
In the case when $\nu_1=\nu_2$ and $K$ commutes with $A$ taking sum of equations 
\eqref{wave-eq1as} and \eqref{wave-eq1bs}
we can find that
\begin{equation}\label{first+w}
 \sup_{t\in \R}\|A(u(t)+v(t))\|\le C(R_\beta) 
\end{equation}
for every trajectory $(u(t);v(t);u_t(t);v_t(t)$ from the attractor $\fA^{\al,\vka}$ with $(\al;\vka)\in\Lambda_\beta$.
Taking the difference of
\eqref{wave-eq1as} and \eqref{wave-eq1bs}
we obtain that
\begin{equation}\label{first+z}
 \sup_{t\in \R}\|Az(t)+2\vka Kz(t)\|\le C(R_\beta)~~\mbox{with}~~ z(t)=u(t)-v(t)
\end{equation}
for fixed $\al$. Since $A$ and $K$ commutes,
\begin{align*}
	\|Az+2\vka Kz\|^2 & =
	\|Az\|^2+ 4\vka^2\|Kz\|^2+4\vka (Az,Kz)
	\\ &
	=	\|Az\|^2+ 4\vka^2\|Kz\|^2+4\vka \|A^{1/2}K^{1/2}z\|^2\ge\|Az\|^2.   
\end{align*}
Thus \eqref{first+w} and \eqref{first+z} yield additional estimate on the attractor:
 \begin{equation*}
 \sup_{t\in \R}\left\{\|Au(t)\|+\|A v(t)\|\right\}\le C 
\end{equation*}
with the constant $C$ independent of $\vka$.
This provides us with compactness of $U^n(t)$
in the space $C([a,b];\bar{H}^{1-\eps})$ for every $\eps>0$ and makes it possible to improve the statement in \eqref{u-u-upper-inf}.
\end{proof}
\begin{remark}
{\rm
\begin{itemize}
  \item 
The result in \eqref{u-u-upper-inf} means that the attractor $\fA^{\al,\vka}$ becomes ``diagonal''  in the limit of large intensity parameter $\vka$  with fixed or even absent interaction in velocities. 
Thus the components of the system becomes   synchronized in this limit at the level of global attractors. In particular, this  implies that every solution $U(t)=(u(t);v(t))$ to \eqref{wave-eq1s}  demonstrates  the following synchronization phenomenon:
\[
\forall\, \eps>0~\exists\, \vka_*:~~
\limsup_{t\to+\infty}
\left[ \|u_t(t)-v_t(t)\|^2+ \|A^{1/2}u(t)-v(t))\|^2\right]\le \eps, ~~\forall\, \vka\ge \vka_*.
\]
\item
The same conclusion as in \eqref{u-u-upper-inf} can  be  obtain in the limit 
$(\al;\vka)\to +\infty$ inside of $\Lambda_\beta$ for some $\beta$. However as one can see from \eqref{u-u-upper-inf} large $\al$ is not necessary for asymptotic synchronization. This observation improves the result established in \cite{Hale-jdde97} for finite-dimensional systems
which requires for synchronization both parameters $\al$ and $\vka$ to be large.
 \item The possibility to obtain synchronization for fixed small $\vka$ and large $\al$ is problematic. 
 The point is that in the case $\vka=0$ under appropriate requirement on nonlinear forces $B_i$ there are possible two different stationary solutions which demonstrate absence of asymptotic synchronization. 
\end{itemize}
}	
\end{remark}

Now we consider the case of identical interacting subsystems, i.e., we assume that
\begin{equation}\label{wave-lim-data-id}
  \nu_1 =\nu_2\equiv \nu,~~
   D_1=D_{2}\equiv D,~~ B_1(w) =B_2(w)\equiv B(w).
 \end{equation}
In this case we observe asymptotic synchronization for finite values of $\vka$.

\begin{theorem}\label{th:sync-ident}
Let the hypotheses of Theorem~\ref{th:att-sync} and also relations
\eqref{wave-lim-data-id} be in force.
Assume that $\al\in [\bar\al,\al_*]$ 
for some fixed $\al_*$.
Let
\begin{equation*}
s_\vka= \inf\big\{ \nu (Aw,w) +\vka (Kw,w)\, : 
w\in H^{1/2},~\|w\|=1 \big\}   
\end{equation*}
There exists
$s_*=s_*(\bar\al,\al_*)$  and $\omega>0$ such that
under the condition\footnote{One can see that $s_\vka\ge \vka \cdot \inf\, {\rm spec}(K)$. Thus if $K$ is not degenerate, then $s_\vka\to+\infty$ as
$\vka\to+\infty$. 
}
$s_\vka\ge s_*$
the property of asymptotic exponential synchronization holds, i.e.,
\begin{equation}\label{asi-sync}
\lim_{t\to\infty}\left\{  e^{\omega t}
\left[ \|u_t(t)-v_t(t)\|^2+ \|A^{1/2}u(t)-v(t))\|^2\right]\right\}=0 
\end{equation}
 for every solution $U(t)=(u(t);v(t))$
to \eqref{wave-eq1s}.
In this case
$\fA^{\al,\vka}\equiv\widetilde{\fA}$
for all $\vka$ such that $s_\vka\ge s_*$.
\end{theorem}
\begin{proof} 
In the case considered $w=u-v$ satisfies the equation
\begin{equation*}
    w_{tt} +\nu A w+ Dw_t  +2\al K w_t +2\vka K w +B(u)-B(v)=0,
    ~~  w(0)=w_0,~~w_t(0)=w_1.
 \end{equation*}
 where $w_0=u_0-v_0$ and $w_1=u_1-v_1$.
\par 
We consider the case of the critical nonlinearity, the subcritical case   
is much simpler.
In the former case we 
use the representation \eqref{newsm} with $z=w$. Since $B_1=B_2=B$, below we omit the subscript $i$.
It follows from \eqref{pi-2-bnd} and \eqref{pi-3-bnd} that the variables $Q$ and $R$ defined 
in \eqref{q-new} and \eqref{r-new} admits the estimates
\[
|Q(t)|\le C_R\|A^\sigma w(t)\|^2
~~\mbox{and}~~
|R(t)|\le C_R(\|u_t(t)\|+\|v_t(t)\|) 
\|A^{1/2} w(t)\|^2
\]
under the condition
\begin{equation}\label{dis-in}
\|A^{1/2} u(t)\|^2+\|A^{1/2} v(t)\|^2\le R^2  
\end{equation}
with $R$ and thus $C_R$ independent of $\al$ and $\vka$.

We consider a Lyapunov type function of the form
\[
\Psi(t)= \widetilde{E}(t)+\Phi(t),
\]
where
\[
\widetilde{E}(t)=\hf \left(\|w_t(t)\|^2+ \nu  \| A^{1/2} w(t)\|^2\right)+ Q(t)\
\]
and 
\[
\Phi(t)=\eta (w,w_t)+\mu (Kw,w),
\]
where $\eta$  is a positive constant which will be chosen later and $\mu=\vka+\eta \al$.
By uniform dissipativity of the system $(\cH,S_t^{\al,\vka})$ we can assume that that \eqref{dis-in} holds with the same $R$ as in \eqref{set-B-uni} for all $t\ge t_*$.

One can see that there exists $0<\eta_0<1$ and $\beta_i>0$ independent of  $(\alpha;\vka)$
such that 
\begin{equation*}
 \beta_0\big[ E_0(t)+\vka \| K^{1/2} w(t)\|^2- c_R \|w(t)\|^2 \big] \le \Psi\le 
 \beta_2\big[ E_0(t)+ c_R \|w(t)\|^2  \big]+\mu  \| K^{1/2} w(t)\|^2,
\end{equation*}
for all $t\ge t_*$ and $0<\eta<\eta_0$,
 where 
\[
E_0(t)=\hf \left(\|w_t(t)\|^2+  \nu\|A^{1/2} w(t)\|^2\right).
\]
Now on strong solutions we calculate the derivative
\begin{align*}
\frac{d\Psi}{dt}= & -((D+\alpha K)w_t,w_t) -R(t) \\
& +\eta\big[ \|w_t\|^2 -(D w_t,w)-\nu (A w,w) -2\vka (K w,w)-(B(u)-B(v),w)\big]
\end{align*}
Since $D$ is bounded from $H^{1/2}$ into 
$H$, we obtain that
\[ 
|(Dw_t,w)|\le \eps\|A^{1/2} w\|^2 +C\eps^{-1}\|w_t||^2,~~~\forall  \eps>0.
\]
Thus
there exist $b_i>0$ independent
of $(\alpha,\vka)$ such that
\begin{align*}
\frac{d\Psi}{dt}\le  & -\big[((D+\alpha K)w_t,w_t) 
-b_1\eta\|w_t\|^2\big]	+ C_R(\|u_t\|+\|v_t\|)\|A^{1/2}w\|^2 \\
& -b_2\eta\big[E_0(t) +\vka \|K^{1/2}w\|^2 \big] +\eta c_R \|w(t)\|^2,
\end{align*}
Fixing $\al$ and taking $s_\vka$ large enough
we obtain that
  \[
\frac{d\Psi}{dt}+ \omega \Psi(t) - C_R(\|u_t\|^2+\|v_t\|^2)\|A^{1/2}w\|^2
	\le 0,~~ t\ge t_*,
\] for some $\omega,C>0$. Using finiteness of the dissipativity integrals:
\[
\int_0^\infty(\|u_t\|^2+\|v_t\|^2)dt<\infty,
\]
we obtain \eqref{asi-sync}.
The equality  $\fA^{\al,\vka}\equiv\widetilde{\fA}$ for the attractors follows from
\eqref{asi-sync}.
\end{proof}
If $B$ is critical, but not satisfies the structural hypothesis in Assumption~\ref{as-sm-tr} we can still guarantee asymptotic exponential synchronization.
However in this case we need additional condition that the damping parameter $\al$ is large enough and $K$ is not    
degenerate. If $B$ are globally Lipschitz we can even
avoid the requirement of dissipativity of the system. 

\begin{remark}\label{re:N-eqs2}
{\rm 
	The results similar to 
	 Theorems \ref{th:att-sync} and \ref{th:sync-ident} can be also established for
	 $N$ coupled second order in time equations of the form
	 \begin{align*}
&   u^1_{tt} +\nu_1 A u^1+  D_{1}u^1_t +\al  K(u^1_t-u^2_t)
    + \vka  K(u^1-u^2) +B_1(u^1)=0,
 \\
&   u^j_{tt} +\nu_j A u^j+  D_{i}u^j_t -\al K(u^{j+1}_t -2u^j_t+u^{j-1}_t)
    - \vka K(u^{j+1} -2u^j+u^{j-1})+ B_j(u^j)=0,
    \\ & \hskip 12 cm
 ~~j=2,\ldots,N-1, \\
 &   u^N_{tt} +\nu_N A u^N+  D_{N}u^N_t +\al  K(u^N_t-u^{N-1}_t)
    + \vka  K(u^N-u^{N-1}) +B_N(u^N)=0,
\end{align*}	
This system can be reduced to \eqref{ver-problem-kappa} with 	
\[
\cD_0=\left(\begin{matrix}
D_{1} &  0  & 0 & \ldots & 0\\
0  &  D_{2} & 0& \ldots & 0\\
0  &  0 & D_{3}& \ldots & 0\\
\vdots & \vdots & \vdots& \ddots & \vdots\\
0  &  0& 0       &\ldots & D_{N}
\end{matrix}\right),~~~
\cK = \left(\begin{matrix}
1 &  -1  & 0 & \ldots & 0\\
-1  &  2 & -1& \ldots & 0\\
0  &  -1 & 2& \ldots & 0\\
\vdots & \vdots & \vdots& \ddots & \vdots\\
0  &  0& 0       &\ldots & 1
\end{matrix}\right) K.
\]
Thus the general results of this section can be applied with the same  hypotheses concerning the operators $A$, $D_{i}$, $K$ and $B_i$. 
We note that the energy for for this $N$ coupled
model has the form
\[
\cE=
\sum_{j=1}^N\left[\frac12\left( \|u^j_t\|^2 + \nu_j\|A^{1/2} u^j\|^2\right)+\Pi_j(u^j)
\right] +\frac{\vka}{2} \sum_{j=1}^{N-1}
\|K^{1/2}(u^{j+1}-u^j)\|^2.
\]
In the ODE case ($\nu_i\equiv 0$, $K=id$) synchronization for this model  was considered in \cite{Hale-jdde97} with assumption that {\em both} $\alpha$ and $\vka$ become  large or even tend to infinity. Our approach allows us observe asymptotic synchronization 
for fixed $\al$ and in the limit  $\vka\to+\infty$ (for identical subsystems it is sufficient to assume that  $\vka$ is large enough). The limiting (synchronized) regime is described by problem \eqref{wave-eq1as-lim}
 with
 \begin{equation*}
   \nu=\frac{1}{N} \sum_{j=1}^N\nu_j,~~
   D=\frac{1}{N} \sum_{j=1}^N D_j,~~ B(w)=\frac{1}{N} \sum_{j=1}^N(B_j(w).
 \end{equation*}
In the case of a plate with the Berger nonlinearity the same result was obtained in \cite{naboka08} with $D_{0j}=d_j \cdot id$, $\al=0$ and
$K=id$.
}
\end{remark}

\subsection{On synchronization by means of  finite-dimensional coupling}

One can see from the argument given in Theorem \ref{th:sync-ident} that
 asymptotic synchronization 
can be achieved even with finite-dimensional coupling operator. Indeed, 
the only condition we need  is 
\[
\nu (Aw,w)+\vka (Kw,w)\ge c \|w\|^2,~~\forall\, w\in H^{1/2},
\]
with appropriate $c$ depending on the size of an absorbing ball. As it was already mentioned  
if $K$ is a strictly positive operator, then the requirement $s_\vka\ge s_*$ 
holds for large intensity parameter $\vka$.
However it is not necessary to assume non-degeneracy of the operator $K$ to guarantee large $s_\vka$. For instance,
if $K=P_N$ is the orthoprojector onto Span${}\{ e_k:\; k=1,2,\ldots, N \}$, then
\begin{align*}
\nu (Aw,w) + & \vka (Kw,w)\ge \sum_{k=1}^N(\nu\lambda_k+\vka)|(w,e_k)|^2+\nu\sum_{k=N+1}^\infty \lambda_k |(w,e_k)|^2 \\
\ge & 	(\nu\lambda_1+\vka)\sum_{k=1}^N|(w,e_k)|^2+ \nu\lambda_{N+1}\sum_{k=N+1}^\infty |(w,e_k)|^2 \ge \min\{\nu\lambda_1+\vka,\nu\lambda_{N+1}\}\|w\|^2.
\end{align*}
Thus if $\vka\ge \nu(\lambda_{N+1}-\lambda_1)$,
then we can guarantee large 
$s_\vka$ by an appropriate choice of $N$.
\par 
This fact admits some generalization which
based on an assumption  that $K$ is
 a ``good'' approximation (in some sense) for a strictly positive operator.
\par 
Let $V\subset H$ be  separable Hilbert spaces and
$K, L$ be  linear operators from $V$ into $H$.
Assume that $L$ is a strictly positive  on $V$, i.e.,
there exists $a_L>0$ such that 
\[
(Lu,u)\ge a_L\|u\|_H^2,~~u\in V.
\] 
We introduce  the value
$$
e(L,K)\equiv e^H_V (L,K) =\sup\{\Vert L u-Ku\Vert_H\; :\; \Vert u\Vert_V\le 1\}.
$$
In the case when $L=id$ is the identity operator this value is known (see \cite{Au72}) as
  the {\em global approximation error}
in $H$ arising in the approximation of
elements $v\in V$ by elements $Kv$. 
\par 
Now we take $V=H^{1/2}$. It follows from the definition that  
\begin{equation*}
  \Vert L u-Ku\Vert_H\le e(L,K)\|A^{1/2}u\|,~~ u\in H^{1/2}.
\end{equation*}
In this case we obtain
\begin{align*}
\nu (Aw,w)+\vka (Kw,w)= &\nu\|A^{1/2}w\|^2 +\vka(Lw,w) + \vka((K-L)w,w)
 \\
\ge &\nu \|A^{1/2}w\|^2 +a_L \vka\|w\|^2 - \vka e(L,K)\|w\|\|A^{1/2}w\|
\\
\ge & \left(a_L \vka - \frac{\vka^2 e^2(L,K)}{4\nu} \right)\|w\|^2
\end{align*}
Thus according Theorem~\ref{th:sync-ident} under the condition
\[
a_L \vka - \frac{\vka^2 e^2(L,K)}{4\nu} \ge s_*
\]
we have asymptotic exponential synchronization. The latter inequality is valid, when $e^2(L,K)\le 2\nu  a_L\vka^{-1}$ and
$\vka\ge 2s_* a_L^{-1}$ for instance. So $e(L,K)$ should be small and $\vka$ large.
We note that in the case $L=id$ and $K=P_N$
we have $e(L,K)=\lambda_{N+1}^{-1/2}$ and $a_L=1$. So the inequalities above can be realized for some choice $\vka$ and $N$.

Now we describe another situation
where synchronization can be achieved 
 with finite-dimen\-sional coupling.
For this we use  interpolation operators related with a finite family $\cL$
of linear continuous functionals $\{l_j: j=1,\ldots,N\}$ on $H^{1/2}$.
Following \cite{Chu99} 
(see also \cite{cl-mem,cl-book})
we introduce the 
notion of {\it completeness defect} of a set ${\cal L}$ of linear functionals
on $H^{1/2}$ with respect to $H$ by the formula
\begin{equation}\label{7.8.21}
\epsilon_{\cL}\equiv \epsilon_{\cal L}(H^{1/2},H) =\sup\big\{ \parallel w \parallel_{H}  :
w\in H^{1/2},\, l (w)=0,\, l\in {\cal L},\, \parallel w \parallel_{1/2} \le 1
\big\}\;.
\end{equation}
It is clear
that $\epsilon_{{\cal L}_1} \ge\epsilon_{{\cal L}_2}$
provided Span${\cal L}_1\subset{\rm Span}{\cal L}_2$ and
$\epsilon_{\cal L}=0$ if and only if the class of functionals
${\cal L}$ is complete in $H^{1/2}$, i.e,  the property $l(w)=0$ for all
$l\in {\cal L} $ implies $w=0$.
For further properties of the completeness
defect we refer to \cite[Chapter 5]{Chu99}.
\par
We define  the class $\cR_{\cL}$ of so-called interpolation operators which are related with
the set of functionals  given.
We say that a operator  $K$ belongs to
${\cal R}_{\cal L}$ if it has 
 the form
\begin{equation}\label{R-op}
Kv=\sum_{j=1}^N l_j(v) \psi_j, ~~~ \forall\, v\in H^{1/2},
\end{equation}
where $\{\psi_j\}$ is an arbitrary finite set of elements from $H^{1/2}$.
An operator $K\in \cR_\cL$ is called {\em Lagrange } interpolation operator, if it has form
\eqref{R-op} with $\{\psi_j\}$ such that $l_k(\psi_j)=\delta_{kj}$. In
the case of Lagrange operators we have that $l_j(u-Ku)=0$ and thus 
\eqref{7.8.21} yields 
\[
\|u-Ku\|\le \epsilon_\cL\|v\|_{1/2},~~ v\in H^{1/2}.
\]
Hence the smallness of the completeness defect is important requirement from point of view of synchronization.
We refer to \cite[Chapter 5]{Chu99}
for properties of this characteristic 
and for description of sets of functionals with small $\epsilon_\cL$.
The simplest example is modes.
 In this case ${\cal L}=\{ l_j(u)=(u,e_j)\; :\; j=1,2,\ldots,N\}$,
where $ \{ e_k \}$ are eigenfunctions of the operator $A$. The 
 operator $P_\cL$  given by
\begin{equation*}
P_\cL v=\sum_{j=1}^N (e_j,v) e_j, ~~~ \forall\, v\in H^{1/2},
\end{equation*}
is the Lagrange  interpolation operator. 
Moreover,
$\epsilon_{\cal
L}=e(id, P_\cL)=\la_{N+1}^{-1/2}$.
Thus the completeness defect (and the global approximation error)
can be made small after an appropriate choice of $N$. This shows that the situation with $K=P_N$ can be included in the general framework.
\par  
Unfortunately in the general case an interpolation operator of the form
\eqref{R-op} is not symmetric and positive. Therefore  formally we cannot
apply the result on uniform dissipativity  with $K$ of the form \eqref{R-op}.
The situation requires a separate consideration and possibly another set of hypotheses concerning the model. We plan to provide with the corresponding analysis in future.
Here we give only one particular result in this direction.
\par
We consider the following version of the equations \eqref{wave-eq1s}
 \begin{subequations}\label{eq-lagr-s}
 \begin{equation}\label{eq-lagr-s1}
    u_{tt} +\nu A u+ D u_t 
    +\vka  K(u-v) +B(u)=0,
 \end{equation}
 \begin{equation}\label{eq-lagr-s2}
    v_{tt} +\nu A v+ Dv_t 
    +\vka K(v-u) +B(v)=0,
 \end{equation}
 \begin{equation}
 u(0)=u_0,~~u_t(0)=u_1,~~  v(0)=v_0,~~v_t(0)=v_1.
\end{equation}
 \end{subequations}
\begin{theorem}\label{th:lagrange}
Assume that $A$ satisfies Assumption \ref{A1-sync2}(i),
$B(u)$ is globally Lipschitz and  subcritical, i.e.,
\begin{equation*}
\exists\, \sigma_0<1/2:~~\|B(u_1) -B(u_2)|~\leq~L_B \|\cA^{\sigma_0}( u_1-u_2)\|,\quad
\forall  u_i \in H^{1/2}.
\end{equation*}	
Let $D : H^{1/2}\mapsto H$ be strictly positive operator and $K\equiv K_\cL$
be a Lagrange interpolation operator for
some family $\cL$ of
linear continuous functionals $\{l_j: j=1,\ldots,N\}$ on $H^{1/2}$.
Then for every initial data $U_0=(u_0;v_0)\in \bar{H}^{1/2}$ and
$U_1=(u_1;v_1)\in \bar{H}$ problem \eqref{eq-lagr-s} has unique generalized solution $U(t)=(u(t);v(t))$ and there exist $\vka_*>0$ and $\epsilon_*(\vka)$ 
such that for every $\vka\ge \vka_*$ 
and $\epsilon_\cL\le \epsilon_*(\vka)$
the solution $U(t)=(u(t);v(t))$ is asymptotically synchronized, i.e.,
relation \eqref{asi-sync} holds with some $\omega>0$. 
\end{theorem}
\begin{proof}
This is globally Lipschitz case and therefore the well-posedness easily follows from \cite[Theorem 1.5]{cl-mem}.
\par 
Following the same idea as in Theorem~\ref{th:sync-ident}  	
we find that the difference
 $w=u-v$ satisfies the equation
\begin{equation*}
    w_{tt} +\nu A w+ Dw_t  +2\vka  w +G_{\vka,\cL}(u,v)=0,
    ~~  w(0)=w_0,~~w_t(0)=w_1.
 \end{equation*}
 where $w_0=u_0-v_0$, $w_1=u_1-v_1$ and
 \[
 G_{\vka,\cL}(u,v)=-2\vka (id-K_\cL)(u-v)  +B(u)-B(v).
 \]
We have the obvious estimate 
 \[
\| G_{\vka,\cL}(u,v)\|\le  2\vka\epsilon_\cL\|A^{1/2}w\|  +L_B \|\cA^{\sigma_0}w\|\le  2(\vka\epsilon_\cL+\delta)\|A^{1/2}w\|  +C_{\delta} \|w\| 
 \]
for all $\delta>0$. This implies 
 \[
| (G_{\vka,\cL}(u,v),w)|\le   (\vka\epsilon_\cL+\delta)^2\|A^{1/2}w\|^2  +C_{\delta} \|w\|^2,~~\forall\, \delta>0, 
 \]
and
\[
| (G_{\vka,\cL}(u,v),w_t)|\le \mu\|w_t\|^2  + \frac{2(\vka\epsilon_\cL+\delta)^2}{\mu}\|A^{1/2}w\|^2  +C_{\delta,\mu } \|w\|^2,~~\forall\,\mu,\delta>0. 
 \]
Now as in the proof of Theorem \ref{th:sync-ident} we can use Lyapunov type functional
\[
\Psi(w,w_t)= \hf \left(\|w_t\|^2+ \nu  \| A^{1/2} w\|^2\right)+
\vka\|w_t\|^2 +
\eta (w,w_t),
\]
with $\vka$ large, $\vka\epsilon_\cL$ small 
and with an appropriate choice of $\mu$ and $\delta$.
\end{proof}

\section{Applications}\label{Sect-appl}
In this section we shortly outline possible application of the general results presented above.
\subsection{{Plate models}}
We first   consider  the plate models
with coupling via elastic (Hooke type) links.
Namely, we consider the following PDEs
\begin{subequations}\label{eq-plate}
	\begin{equation}\label{eq-plate1}
u_{tt} +\ga_1u_{t}+ \Delta^2u+ \vka K (u-v)+
\f_1(u)=f_1~~\mbox{in}~ \Omega\subset\R^2,
\end{equation}
\begin{equation}\label{eq-plate2}
v_{tt} +\ga_2v_{t}+ \Delta^2v+ \vka K (v-u)+
\f_2(u)=f_2~~\mbox{in}~ \Omega\subset\R^2,
\end{equation} with  the hinged boundary 
conditions
\begin{equation}
 u=\Delta u=0,~~v=\Delta v=0~~\mbox{on}~ \partial\Omega, 
\end{equation} 
\end{subequations}
 where $K$ is an operator which will be specified later. 
The nonlinear force term $\f_i(u)$
 can take one of the following forms:
\begin{itemize}
  \item {\sl Kirchhoff model}:
$\f\in {\rm Lip_{loc}}(\R)$ fulfills the  condition \begin{equation*}
    \underset{|s|\to\8}{\liminf}{\frac{\f(s)}{s}}>-\l_1^2,
\end{equation*}  where $\l_1$ is the first eigenvalue of the
Laplacian with the Dirichlet boundary conditions.

  \item {\sl Von Karman model:} $\f(u)=[u, v(u)+F_0]$, where $F_0$ is a given function
  in $H^4(\Omega)$ and
the bracket $[u,v]$  is given by
\begin{equation*}
[u,v] = \partial ^{2}_{x_{1}} u\cdot \partial ^{2}_{x_{2}} v +
\partial ^{2}_{x_{2}} u\cdot \partial ^{2}_{x_{1}} v -
2\cdot \partial ^{2}_{x_{1}x_{2}} u\cdot \partial ^{2}_{x_{1}x_{2}}
v .
\end{equation*} 
The Airy stress function $v(u) $ solves the following  elliptic
problem
\begin{eqnarray*}
\Delta^2 v(u)+[u,u] =0 ~~{\rm in}~~  \Omega,\quad \frac{\partial v(u)}{\partial
n} = v(u) =0 ~~{\rm on}~~  \partial\Omega.
\end{eqnarray*}
Von  Karman equations are well known in nonlinear elasticity and
 describes nonlinear oscillations of a
plate accounting for  large displacements, see \cite{cl-book,Lio69}  and the
references therein.
  \item {\sl Berger Model:} In this case the feedback force has the form
  $$
  \f(u)=- \left[ \kappa \int_\Omega |\nabla u|^2 dx-\G\right] \Delta u,
$$
where $\kappa>0$ and $\G\in\R$ are parameters, for some details and  references see
\cite[Chap.4]{Chu99}.
\end{itemize}
In all these cases we have that $H=L^2(\Omega)$ and
\[
A u= \Delta^2 u,\quad u\in \sD(\cA)=
\left\{ u\in H^4(\Omega) : u=\Delta u=0 ~~{\rm on}~~\partial\Omega\right\}.
\]
It is clear that $A$ satisfies Assumption~\ref{A1-sync2} (i)
and $\sD(A^{1/2})=H^2(\Omega)\cap H^1_0(\Omega)$.
\par

The nonlinear force $\f$ in the Kirchhoff model is subcritical 
with respect to the energy space (it is locally Lipschitz from $H^{1+\delta}(\Omega)$ into $L_2(\Omega)$ for every $\delta >0$). In contrast,
the von Karman and Berger nonlinearities are {\em critical} (they are 
 locally Lipschitz mappings from $H^2(\Omega)$ into $L_2(\Omega)$ which are not compact
on $H^2(\Omega)$). 
Other requirements concerning  the corresponding forcing terms $B_i$ can be verified  in the standard way. For details
  we refer to \cite{kolbasin}  for the Kirchhoff forces, to \cite[Chapter 6]{cl-mem} for the von Karman model and to \cite[Chapter 7]{cl-mem}
for the case of Berger plates.
\par
The interaction operator $K$ can be  of the following forms $K=id$ and $K=-\Delta$,
or even $K=A^\sigma$ with $0<\sigma<1/2$.
In the purely Kirchhoff case with globally Lipschitz functions $\f_i$
we can also use the Lagrange interpolation operator with respect to \textit{nodes}, i.e., with respect to the family of functionals
\[
l_j(w)= w(x_j),~~\mbox{where}~~x_j\in \Omega,~~j=1,\ldots,N,
\]
with appropriate\footnote{  
For details concerning smallness of the 
corresponding completeness defect we refer to \cite[Chapter 5]{Chu99}.} 
choice of nodes $x_j$.
This means that two Kirchhoff plates can be synchronized by finite number of point links.

We note that in the case when both $\f_1$ and $\f_2$ are Berger nonlinearities (possibly with different parameters)
the results on synchronization with $K=id$  can be found in \cite{naboka07}, see also \cite{naboka08,naboka09}.

We also mention some other plate models  for which  the abstract results
established can be applied:
\begin{itemize}
  \item First of all we can consider the plates with other (self-adjoint) boundary conditions
  such as clamped and free and also combinations of them
  (for a discussion of these boundary conditions in the case of 
  nonlinear plate models we refer to \cite{cl-book}).
  \item The plate models with rotational inertia
  can be included in the framework presented. Instead of \eqref{eq-plate}
coupled dynamics  in these models  can be described by
equations of the form
\begin{subequations}\label{eq-rot}
	\begin{equation}
(1-\ga\Delta)u_{tt} +\mu (1-\ga\Delta)u_{t}+ \Delta^2u+ \vka K (u-v)+
\f_1(u)=f_1~~\mbox{in}~ \Omega\subset\R^2,
\end{equation}
\begin{equation}
(1-\ga\Delta)v_{tt} +\mu(1-\ga\Delta)v_{t}+ \Delta^2v+ \vka K (v-u)+
\f_2(u)=f_2~~\mbox{in}~ \Omega\subset\R^2.
\end{equation}
 \end{subequations}
Here $\gamma$ is positive.
It is convenient  to rewrite  \eqref{eq-rot}
 as  equations in $H=H^1_0(\Omega)$ (equipped with
 the inner product $(u,v)_H\equiv((1-\ga\Delta)u,v)_{L^2(\Omega)}$) in the form \eqref{wave-eq1s}
 with the operator $A$ generated by the form
$a(u,v)=(\Delta u,\Delta v)_{L^2(\Omega)}$ on $H^2_0(\Omega)$.
\item  We can also include into consideration
the plates with Kirchhoff--Boussinesq forces of the form
\[
\f(u)=-{\rm div}\!\left\{ |\nabla u|^2 \nabla u\right\}+ a |u|^q u
\]
with $a,q\ge 0$, see \cite[Chapter 7]{cl-mem} and also \cite{cl-kb,cl-kb2} concerning models with this force.
\end{itemize}

\subsection{Coupled wave equations}\label{sec:cpl-wave}

 In the case of coupled  wave equations   \eqref{3d-wave}
the standard (critical) hypotheses concerning the source term $\f\in C^2(\R)$
 in the 3D case have the form:
\begin{equation*}
  \underset{|s|\to\infty}{\liminf}\left\{\f_i(s)s^{-1}\right\}>-\l_1,\quad
|\f_i''(s)|\le C(1+|s|),\; s\in\R,
\end{equation*}
 where $\l_1$ is the first eigenvalue of the minus
Laplacian with the Dirichlet boundary conditions, see, e.g., \cite[Chapter 5]{cl-mem} for details.
\medskip\par  

We can also consider several versions of damped sine-Gordon equations.
These are used to model the dynamics of Josephson junctions driven by a
source of current (see, e.g., \cite{temam}  for  comments and  references).
For instance, we can consider the system\footnote{For simplicity we discuss a symmetric coupling of identical systems only.}
\begin{subequations}\label{3d-wave-sin-G}
\begin{align}
& u_{tt}+\gamma u_t -\Delta u+ \beta u+ \vka (u-v) +\la \sin u=f(x), \\
& v_{tt}+\gamma v_t -\Delta v+ \beta v+ \vka (v-u) +\la \sin v=f(x),
\end{align}
in a smooth domain $\Omega\subset \R^d$ and equipped with the Neumann boundary conditions 
\begin{equation}
\frac{\partial u}{\partial n}\Big|_{\partial\Omega}=0,~~ 
\frac{\partial v}{\partial n}\Big|_{\partial\Omega}=0.
\end{equation} 
\end{subequations}
 It is easy to see that in the case 
 of the Dirichlet boundary conditions
 the general theory developed above can be applied.
 The same is true when $\beta>0$.
  In the case $\beta=0$ the situation is more complicated because the corresponding operator $A$
  is $-\Delta$ on the domain
  \[
  \cD(A)=\left\{ u\in H^2(\Omega)\, :~
  \frac{\partial  u}{\partial n}=0~~
 \mbox{on}~~ \partial\Omega\right\}
  \]
 and thus   becomes degenerate.
 So we concentrate on the case $\beta=0$.
 \par 
 It is convenient to introduce new variables
 \begin{equation}\label{wz-vari}
 w=\frac{u-v}{2}~~\mbox{and}~~z=\frac{u+v}{2}.
\end{equation}
 In these variables problem \eqref{3d-wave-sin-G} with $\beta=0$ can be written in the form
 \begin{subequations}\label{3d-wave-sG-wz}
\begin{align}
& w_{tt}+\gamma w_t -\Delta w+ 2\vka w +\la \sin w \cos z=0, \label{3d-sG-w-eq} \\
& z_{tt}+\gamma z_t -\Delta z +\la \cos w \sin z=f(x),
\\ &
\frac{\partial w}{\partial n}\Big|_{\partial\Omega}=0,~~ 
\frac{\partial z}{\partial n}\Big|_{\partial\Omega}=0.
\end{align} 
\end{subequations}
The main linear part in \eqref{3d-sG-w-eq}
is not degenerate when $\vka>0$. Therefore the same calculations as in the proof of Theorem~\ref{th:sync-ident} shows that
there exists $\vka_*$ such that
\[ \exists\, \eta>0\, :~~
 \|w(t)\|_{H^1(\Omega)}^2+\|w_t(t)\|^2\le
C_B e^{-\eta t},
~~t >0,
\] 
when $\vka\ge \vka_*$ 
for all initial data from a bounded set $B$
in $H^1(\Omega)\times L_2(\Omega)$.
This means that every trajectory is asymptotically synchronized.
Moreover, it follows from the reduction principle (see \cite[Section 2.3.3]{Chu-dqsds15})
 that the limiting (synchronized) dynamics 
is determined by the single equation
\begin{equation*}
  z_{tt}+\gamma z_t -\Delta z +\la  \sin z=f(x), ~~ 
\frac{\partial z}{\partial n}\Big|_{\partial\Omega}=0.
\end{equation*}
The long-time dynamics of this equation is described in \cite[Chapter 4]{temam}.
 We also refer to \cite{leonov1,leonov2} for the studies  of synchronization phenomena for \eqref{3d-wave-sin-G} in the homogeneous (ODE) case.
\medskip\par 
Another coupled sine-Gordon systems of an interest is the following one
\begin{subequations}\label{3d-wave-sin-G2}
\begin{align}
& u_{tt}+\gamma u_t -\Delta u  +\la \sin (u-v)=f_1(x), \\
& v_{tt}+\gamma v_t -\Delta v +\la \sin(v-u) =f_2(x),
\\ &
u|_{\partial\Omega}=0,~~ 
 v|_{\partial\Omega}=0.
\end{align}
\end{subequations}
Formally this model is out of the scope the theory developed above.
However, using the ideas presented we can answer some questions concerning synchronized regimes.\footnote{In the ODE case 
the synchronization phenomena in \eqref{3d-wave-sin-G2} was studied 
in \cite{leonov1,leonov2}.
}
 \par 
In variables $w$ and $z$ given by   
\eqref{wz-vari} we have equations
\begin{subequations}\label{3d-wave-sG-wz2}
\begin{align}
& w_{tt}+\gamma w_t -\Delta w +\la \sin 2w =g(x), \label{w-sin} \\
& z_{tt}+\gamma z_t -\Delta z =h(x),
\\ &
w|_{\partial\Omega}=0,~~ 
 z|_{\partial\Omega}=0,
\end{align} 
\end{subequations}
where $g(x)=(f_1(x)-f_2(x))/2$ and
$h(x)=(f_1(x)+f_2(x))/2$.
One can see that
\[
\|z_t(t)\|^2+\|\nabla (z(t)-z_*)\|^2
\le C\left(\|z_t(0)\|^2+\|\nabla (z(0)-z_*)\|^2\right) e^{-\omega t},~~ t\ge 0,
\]
for some $C,\omega>0$, where $z_*\in H^1(\Omega)$ solves the Dirichlet problem
\[
-\Delta z =h(x),
~~ 
 z|_{\partial\Omega}=0.
\]
Problem \eqref{w-sin} equipped with
the Dirichlet boundary conditions possesses a compact global  
attractor $\fA$ in the space
$W=H_0^1(\Omega)\times L_2(\Omega)$, see \cite[Chapter 4]{temam}.
Hence
\[
\left(\begin{matrix}
  u(t) -z_* \\
  u_t(t) \\
  v(t)-z_* \\
  v_t(t)
\end{matrix}\right)=
\left(\begin{matrix}
  z(t) -z_* \\
  z_t(t) \\
  z(t)-z_* \\
  z_t(t)
\end{matrix}\right)+
\left(\begin{matrix}
  w(t) \\
  w_t(t) \\
  -w(t) \\
  -w_t(t)
\end{matrix}\right) \longrightarrow
\left\{ \left(\begin{matrix}
  \psi  \\
  -\psi  \\
\end{matrix}\right)\, : \; \psi\in\fA\right\}~~\mbox{as}~~ t\to+\infty
\]
in the space $W\times W$.
Thus we observe some kind of shifted asymptotic anti-phase
synchronization.

%


\appendix

\section{Appendix: Some facts from the theory of dynamical systems}

In this section we collect several  definitions  and general theorems
from the theory of dissipative dynamical systems.
For more complete presentation we refer to one of the monographs
\cite{BV92,Chu99,Ha88,temam}.
\par
Recall that a
{\em dynamical system} (see, e.g., \cite{Chu99,Ha88,temam}) is a pair
$\big(X, S_t\big)$ of a complete metric space $X$ and a family of
continuous mappings $S_t:X\mapsto X,\ t\ge 0$,
 satisfying the semigroup property: (i)
 $ S_{t+\t} = S_t\circ S_\t$ for any $t,\t\ge 0$, and (ii)
$S_0$ is the identity operator.
It is also assumed that $t\mapsto S_tx$ is continuous mapping for every $x\in X$.
\par
A system $\big( X, S_t\big)$ is said to be
{\em dissipative} 
if it possesses a bounded absorbing set $B$.
A closed set
$B\subset X$  is said to be {\em  absorbing} for $S_t$
if for any bounded set $D\subset X$  there exists
$t_0(D)$ such that
$S_t D\subset B$ for all $t\ge t_0(D)$.
If  the phase space $X$ 
is a Banach space, then the radius of a ball containing  an absorbing set
is called  a {\em radius of dissipativity} of the system.
\par
A system $\big( X, S_t\big)$ is said to be \textit{asymptotically smooth}
if for any  closed bounded set $B\subset X$
such that  $S_tB\subset B$ for all $t\ge 0$
there exists a compact set $\cK=\cK(B)$ which  uniformly attracts $B$, i.e.,
$$
\lim_{t\to+\8} \sup\{{\rm dist}(S_ty,\cK):\ y\in B\} = 0.
$$
\par
A system is called \textit{gradient} if it possesses a
\textit{strict Lyapunov function}. The latter  is defined as a
continuous functional $\Phi(x)$ on $X$ satisfying
(i) $\Phi\big(S_tx\big) \leq \Phi(x)$  for all $t\geq 0$ and $x\in X$, and (ii)
 if $\Phi(x)=\Phi(S_tx)$ for all $t>0$, then $x$
is a stationary point of $S_t$, i.e.,
$S_t x=x$ for all $t\ge 0$. 
\par
A \textit{global attractor}
 of a dynamical system  $\big( X, S_t\big)$ is a bounded closed  set $\fA\subset X$
 which is  invariant (i.e., $S_t\fA=\fA$)
 and  uniformly  attracts
all other bounded  sets:
$$
\lim_{t\to\8} \sup\{{\rm dist}(S_ty,\fA):\ y\in B\} = 0
\quad\mbox{for any bounded  set $B$ in $X$.}
$$
It is known \cite{BV92,Chu99,Ha88,temam} that the global attractor consists of all bounded full trajectories. 
We recall that a curve $\{y(t)\, :\, t\in\R\}\subset X$ is said to be a \textit{full trajectory} if $S_t y(\t)=y(t+\t)$
for all $\t\in \R$ and $t\ge 0$.

The standard criterion (see, e.g., \cite{Chu99,Ha88,temam}) 
 for the existence of a global attractor 
is the following assertion.

\begin{theorem}\label{th:main-attractor-a}
Let $(X,S_t)$ be  a dissipative    asymptotically smooth
dynamical system on  a complete metric space $X$.
Then $S_t$ possesses a unique  compact global attractor
$\fA$ such that
\[
\fA=\omega(B_0)=
\bigcap_{t>0}\overline{\bigcup_{\tau\ge t} S_\tau B_0}
\]
for every bounded absorbing set $B_0$.
\end{theorem}
For gradient systems it is also useful
 the following criterion of the existence
of a global attractor  (see, e.g., \cite[Theorem 4.6]{Raugel}).
\begin{theorem}\label{Theorem 2.2.}
Let $\big( X, S_t \big)$ be an asymptotically smooth gradient system
 which has the property that for any bounded set
 $B\subset X$ there exists $\t>0$ such that
 $\ga_\t(B)\equiv\bigcup_{t\geq\t}S_tB$ is bounded. If the set $\cN$ of
stationary points  is bounded, then $\big( X, S_t \big)$ has a
compact global attractor $\fA$ which coincides with the unstable set $\mathbb{M}_+(\cN)$
emanating from $\cN$, i.e., $\fA=\mathbb{M}_+(\cN)$.
\end{theorem}
We recall (see, e.g., \cite{BV92})
 that the \textit{unstable set} $\mathbb{M}_+(\cN)$ emanating from $\cN$ is a subset
 of $X$ such that for each $z\in\mathbb{M}_+(\cN)$
there exists a full trajectory $\{y(t): t\in\R\}$ satisfying
$u(0) = z$ and ${\rm dist}(y(t),\cN) \to  0$ as $t\to -\8$.

\begin{remark}\label{re:th2.2}
{\rm
We note that the hypothesis that $\ga_\t(B)$ is bounded
in Theorem~\ref{Theorem 2.2.} can be changed in the following
(additional) requirements concerning the Lyapunov function $\Phi(x)$: (i)
$\Phi(x)$ is bounded from above on any bounded set, and (ii)
the set $\Phi_R=\{x\in X: \Phi(x)\le R\}$ is bounded for
every $R$ (see, e.g., \cite[Corollary 2.29]{cl-mem}).
}
\end{remark}

An important feature of a global attractor is its (fractal) dimension.
We recall that
the {\em fractal dimension} $\dim^X_f M$ of a compact set $M$ in a complete
metric space $X$ is defined (see, e.g., \cite{temam}) as
\[
\dim^X_fM=\limsup_{\eps\to 0}\frac{\ln N(M,\eps)}{\ln (1/\eps)}\;,
\]
where $N(M,\eps)$ is the minimal number of closed sets in $X$ of
diameter $2\eps$ needed to cover the set~$M$.
\par
Now we state several facts related with 
asymptotically quasi-stable systems
(for details we refer to the recent monograph 
\cite{Chu-dqsds15} and the references therein).

Let $X$ and $Y$ be reflexive Banach spaces; $X$ is compactly embedded
in $Y$. We endow the space $H=X\times Y$ with the norm
\[
\|y \|^2_H=\|u_0\|^2_X+\|u_1\|^2_Y,\quad y=(u_0;u_1).
\]
Assume that
$(H,S_t)$ is a dynamical system   with the evolution operator
of the form
\begin{equation}\label{7.9.1}
S_ty=(u(t); u_t(t)), \quad y=(u_0;u_1)\in H,
\end{equation}
where the function $u(t)$ possesses the property
\[
u\in C(\R_+, X)\cap C^1(\R_+, Y).
\]
A dynamical system  $(H,S_t)$ with an evolution operator of the form \eqref{7.9.1} is said to be \textit{asymptotically  quasi-stable} on a set $\cB\subset H$ if
there exist a compact seminorm $\mu_X(\cdot)$ on the space $X$ and
 nonnegative scalar functions $a(t)$, $b(t)$, and $c(t)$ on $\R_+$
such that
(i) $a(t)$ and  $c(t)$ are locally bounded on $[0,\infty)$,
(ii)~$b(t)\in L_1(\R_+)$ possesses
the property $\lim_{t\to\infty}b(t)=0$,
and  (iii) for every $y_1,y_2\in \cB$ and $t>0$ the following relations
\[
\| S_ty_1-S_ty_2\|^2_H\le  a(t)\cdot \| y_1-y_2\|^2_H
\]
and
\begin{equation}\label{8.4.2mc}
\| S_ty_1-S_ty_2\|^2_H\le  b(t)\cdot \| y_1-y_2\|^2_H+
c(t)\cdot \sup_{0\le s\le t}\left[ \mu_X(u^1(s)-u^2(s))\right]^2
\end{equation}
hold. Here we denote
$S_ty_i=(u^i(t); u^i_t(t))$, $i=1,2$.
The theory of quasi-stable systems was started with \cite{cl-jdde} and has been developed in \cite{Chu-dqsds15,cl-mem,cl-book,cl-hcdte}.
The main outcome of this theory is the following assertion.
\begin{theorem}[Global attractor]\label{th7.9dim} 
Assume that the system $(H,S_t)$ is dissipative
 and asymptotically quasi-stable on a bounded forward invariant absorbing set $\cB$ in $H$.
Then  the  system $(H,S_t)$ possesses a compact
global attractor $\fA$ of finite
fractal dimension.
\end{theorem}
Another consequence of quasi-stability
 is the following assertion 
which states some regularity of the attractor and provides additional bounds for  trajectories in it.

\begin{theorem}[Regularity]\label{th7.9reg}
Assume that the dynamical system $(H,S_t)$ possesses a compact global attractor
$\fA$ and is asymptotically quasi-stable on  $\fA$.
Moreover, we assume that \eqref{8.4.2mc} holds with the function
$c(t)$ possessing the property $c_\infty=\sup_{t\in\R_+}c(t)<\infty$.
 Then any full trajectory $\{ (u(t);u_t(t))\, :\, t\in\R\}$
that belongs to the global attractor enjoys
 the following regularity  properties
\[
u_t \in  L_\infty(\R; X)\cap C(\R; Y),  \quad u_{tt} \in  L_\infty(\R; Y).
\]
Moreover, there exists $R>0$ such that
\[
\| u_t(t)\|_X^{2}+ \| u_{tt}(t)\|_Y^{2}   \le  R^{2}, \quad t\in\R,
\]
where $R$ depends on the constant $c_\infty$, on the seminorm $\mu_X$,
and also on the embedding properties of $X$ into $Y$.
\end{theorem}

In the study of synchronization phenomena we 
 deal with  stability of attractors \wrt interaction intensity. Variations of the intensity parameter  is treated as  a perturbation of the
dynamical system. To describe these variations at the abstract level  we consider
 a family of dynamical systems $(X,S_t^\lambda)$ with
the same  phase space $X$ and with evolution operators
$S_t^\lambda$ depending on a parameter $\lambda$ from a complete metric space
$\Lambda$. The following assertion is proved by  Kapitansky and
 Kostin \cite{kap-kos} (see also \cite{BV92} and \cite{Ha88} for related results).
\begin{theorem}[Upper semicontinuity]\label{t7.2.5} Assume that
a dynamical system  $(X,S^\lambda_t)$
in a complete metric space $X$ possesses  a   compact global attractor $\fA^\lambda$
for every $\lambda\in\Lambda$. Assume that the following conditions hold.
\begin{enumerate}
\item[(i)]
There exists a compact $K\subset X$ such that $\fA^\lambda\subset K$.
\item[(ii)]
If $\lambda_k\to\lambda_0$, $x_k\to x_0$ and $x_k\in \fA^{\lambda_k}$, then
\begin{equation}\label{7.2.6-}
S^{\lambda_k}_\tau x_k\to S^{\lambda_0}_\tau x_0 ~~~\mbox{for some $\tau>0$.}
\end{equation}
\end{enumerate}
Then the family $\{ \fA^\lambda\}$  of attractors  is upper semicontinuous at the
point $\lambda_0$; that is,
\[
 d_X\left\{ \fA^{\lambda}\, |\, \fA^{\lambda_0}\right\}\equiv
 \sup\left\{\dist_X (x, \fA^{\lambda_0})\; :\; x\in \fA^{\lambda} \right\}
\to 0
~~\mbox{as}~~ \lambda\to\lambda_0.
\]
Moreover, if \eqref{7.2.6-} holds for every $\tau>0$, then  the upper limit $\fA(\lambda_0, \Lambda)$ of the attractors
$\fA^{\lambda}$ at  $\lambda_0$ defined by the formula
\[
\fA(\lambda_0, \Lambda)=
\bigcap_{\delta>0}\overline{\bigcup\left\{ \fA^\lambda\; :\lambda\in\Lambda,\;
0<\dist (\lambda,\lambda_0)<\delta\right\}}
\]
is a nonempty compact strictly invariant set lying in the attractor 
$\fA^{\lambda_0}$ and
possessing the property
\begin{equation*}
 d_X\left\{ \fA^{\lambda}\, |\, \fA(\lambda_0, \Lambda)\right\}
\to 0~
~\mbox{as}~~ \lambda\to\lambda_0.
\end{equation*}
\end{theorem}

\end{document}